\pdfoutput=1
\RequirePackage{ifpdf}
\ifpdf % We are running pdfTeX in pdf mode
\documentclass[pdftex]{sigma}
\else
\documentclass{sigma}
\fi

\usepackage[all]{xy}

\DeclareMathOperator{\Ad}{Ad}
\newcommand{\CoAd}{\Ad^\dagger}

\def\d{\mathrm{d}}

\DeclareMathOperator{\id}{id}

\newcommand{\toto}{\rightrightarrows}
\newcommand{\vzero}{\mathbf{0}}

\newcommand{\category}[2]
{
\begin{gather*}
\textit{objects:}\quad #1 \\
\textit{morphisms:}\quad #2
\end{gather*}
}

\def\bbp{\mathbb P}

\def\bbr{\mathbb R}

\def\bbx{\mathbb X}

\def\calv{\mathcal{V}}

\def\ffg{\mathfrak g}
\def\ffh{\mathfrak h}

\def\ffk{\mathfrak k}

\def\ffp{\mathfrak p}
\def\ffq{\mathfrak q}

\numberwithin{equation}{section}

\newtheorem{Theorem}{Theorem}[section]
\newtheorem{Corollary}[Theorem]{Corollary}
\newtheorem{Lemma}[Theorem]{Lemma}
\newtheorem{Proposition}[Theorem]{Proposition}
 { \theoremstyle{definition}
\newtheorem{Definition}[Theorem]{Definition}

\newtheorem{Remark}[Theorem]{Remark} }

\begin{document}

\allowdisplaybreaks

\newcommand{\arXivNumber}{1707.02828}

\renewcommand{\PaperNumber}{021}

\FirstPageHeading

\ShortArticleName{Nonlinear Stability of Relative Equilibria and Isomorphic Vector Fields}

\ArticleName{Nonlinear Stability of Relative Equilibria\\ and Isomorphic Vector Fields}

\Author{Stefan KLAJBOR-GODERICH}
\AuthorNameForHeading{S.~Klajbor-Goderich}

\Address{Department of Mathematics, University of Illinois at Urbana-Champaign,\\ 1409 W. Green Street, Urbana, IL 61801 USA}
\Email{\href{mailto:klajbor2@illinois.edu}{klajbor2@illinois.edu}}
\URLaddress{\url{https://faculty.math.illinois.edu/~klajbor2/}}

\ArticleDates{Received October 31, 2017, in final form March 09, 2018; Published online March 14, 2018}

\Abstract{We present applications of the notion of isomorphic vector fields to the study of nonlinear stability of relative equilibria.
Isomorphic vector fields were introduced by Hepworth [\textit{Theory Appl. Categ.} \textbf{22} (2009), 542--587] in his study of vector fields on differentiable stacks. Here we argue in favor of the usefulness of replacing an equivariant vector field by an isomorphic one to study nonlinear stability of relative equilibria. In particular, we use this idea to obtain a criterion for nonlinear stability. As an application, we offer an alternative proof of Montaldi and Rodr\'iguez-Olmos's criterion [arXiv:1509.04961] for stability of Hamiltonian relative equilibria.}


\Classification{37J25; 57R25; 37J15; 53D20}

\section{Introduction}\label{Introduction}

Relative equilibria of equivariant vector fields and their stability have garnered much interest in the dynamics literature, partly due to their myriad applications in the sciences (see, for example,~\cite{GS02}). In this paper we present an approach to determining the stability of relative equilibria via the notion of isomorphic vector fields introduced by Hepworth~\cite{H09}. In particular, we argue that it can be useful to replace a given equivariant vector field with an isomorphic one for which it is easier to determine stability.

Recall that a relative equilibrium of an equivariant vector field is a point for which the vector field is tangent to the group orbit at that point.
It can be difficult to determine the stability of relative equilibria. Even determining linear stability poses a challenge. For an equilibrium, the Lyapunov stability criterion can guarantee linear stability if all the eigenvalues in the spectrum of the linearization of the vector field have negative real part (see, for example, \cite[Theorem~4.3.4]{AMR88}). In contrast, since the vector field is not necessarily zero at a relative equilibrium, the usual notion of a linearization does not make sense. Thus, we don't immediately have an analogue of the Lyapunov stability criterion.

A construction due to Krupa gives a way to linearize an equivariant vector field near a relative equilibrium and test for linear stability~\cite{K90}. Krupa's construction involves choosing a slice for the action through the relative equilibrium and projecting the vector field onto the slice. The projected vector field has an equilibrium at the original vector field's relative equilibrium, so we can linearize the projected vector field. This construction depends on a choice of slice and projection, but it turns out the real parts of the spectrum of the linearization are independent of these choices \cite[Lemma~8.5.2]{F07}. The Lyapunov stability criterion can then be used to test for linear stability of the equilibrium of the projected vector field. It can be shown that if this is linearly stable it implies the linear stability of the relative equilibrium of the original vector field \cite[Theorem~7.4.2]{CL00}. Furthermore, since the real parts of the eigenvalues of the spectrum are independent of the choices, we can choose any slice and projection to determine linear stability; ideally ones where the spectrum is easier to compute.

Not all stable equilibria are linearly stable, and the same is true of relative equilibria. To use Krupa's construction for nonlinear stability, as well as for other applications, we need to make sense of the choices involved. Hepworth's notion of isomorphism of vector fields is useful for this. Hepworth introduced isomorphic vector fields to define vector fields on differentiable stacks, a~categorical generalization of differentiable manifolds~\cite{H09}. Since differentiable stacks are, in some sense, represented by Lie groupoids, it is not surprising that vector fields on a~stack form a groupoid. This gives rise to a~notion of isomorphism between equivariant vector fields. Lerman used Hepworth's notion of isomorphism of vector fields to revisit Krupa's construction~\cite{L15}. In particular, he showed that the choices of slice and projection lead to isomorphic vector fields.

In this paper we show how considering vector fields up to isomorphism, in the sense of Hepworth, facilitates testing for nonlinear stability.
In Theorem \ref{Thm:MainTheorem1}, which we call here the slice stability criterion, we show that one can determine nonlinear stability of a relative equilibrium by testing for nonlinear stability of the corresponding equilibrium of the projected vector field.
This reduces the problem to the well-studied case of equilibria on a vector space with a representation of a compact Lie group.
In fact, one can test any vector field that is isomorphic to the projected vector field.
Hence, one additionally obtains the freedom to choose a convenient slice, projection, and isomorphism class representative to determine stability.

Hamiltonian relative equilibria are an important case where we may have nonlinear stability but not linear stability.
The integral curves of a Hamiltonian vector field do not exhibit energy dissipation, so we don't expect the relative equilibria to be linearly stable. Lerman and Singer~\cite{LS98} and Ortega and Ratiu~\cite{OR99}, building on work of Patrick~\cite{P91, P95}, showed that the definiteness of the Hessian of an augmented Hamiltonian function implies stability of the Hamiltonian relative equilibrium.
Montaldi and Rodr\'iguez-Olmos extended this criterion, allowing for a wide choice of augmented Hamiltonians to check for stability \cite[Theorem~3.6]{MRO15} (see also \cite[Theorem~2]{MRO11}).
They prove this extension by building on the bundle equations in \cite{RWL02, RdS97, WR01}.
We use Theorem~\ref{Thm:MainTheorem1} to provide an alternative proof of their result.
Our proof is based on the fact that the augmented Hamiltonian vector fields are isomorphic to the original Hamiltonian vector field and that a choice of augmented Hamiltonian is equivalent to a choice of an isomorphism class and a representative.

\subsection{Organization of the paper}

In Section~\ref{IsosAndRelEqs}, we present Hepworth's groupoid of equivariant vector fields in the context of Lie group actions, as well as the corresponding notion of isomorphism of equivariant vector fields. We also present an equivalent formulation of the results in \cite{L15}, and provide some general background and results.

In Section~\ref{Stability}, we prove a test for nonlinear stability of relative equilibria, Theorem~\ref{Thm:MainTheorem1}, which we call here the slice stability criterion. This is our main theorem on the nonlinear stability of relative equilibria.
We also show how isomorphisms of equivariant vector fields and one of the functors involved in the slice stability criterion preserve the stability of relative equilibria.

In Section~\ref{HamStability}, we apply the slice stability criterion to obtain a proof of the result of Montaldi and Rodr\'iguez-Olmos (Theorem~\ref{Thm:MROTheorem}). We use the Marle--Guillemin--Sternberg normal form \smash{\cite{GuS84, Ma85}} in this proof. In Section~\ref{MGSNormalForm}, we reduce the general case to the normal form computation.

\subsection{Notation and conventions}

Throughout the paper we will assume all manifolds are Hausdorff. We will denote Lie groups with uppercase Latin letters, their Lie algebras with the corresponding lowercase fraktur letter, and the duals of these Lie algebras by adding a star superscript. The {\it adjoint representation} of a Lie group on its Lie algebra will be denoted by $\Ad$, while its {\it coadjoint representation} on the dual of the Lie algebra will be denoted by~$\CoAd$. Given an action of a Lie group on a manifold, the {\it stabilizer subgroup} of a point $m$ will be denoted by the same letter as the group but with the point as a subscript (e.g.,~$G_m$). The Lie algebra of the stabilizer will also carry the point as a subscript (e.g.,~$\ffg_m$).

The vector space of smooth vector fields on a manifold $M$ will be denoted by $\Gamma(TM)$. Given a~diffeomorphism $f\colon M\to N$ between two manifolds, we will denote the corresponding {\it pushforward of vector fields along $f$} by $f_*\colon \Gamma(TM)\to\Gamma(TN)$ and the {\it pullback of vector fields along~$f$} by~$f^*\colon \Gamma(TN)\to \Gamma(TM)$. We will refer to both embedded and regular submanifolds. Recall an {\it embedded submanifold} of a manifold $M$ is a pair consisting of a manifold $N$ and a smooth embedding $f\colon N\to M$, whereas a {\it regular submanifold} of a manifold~$M$ consists of a subset~$A$ of~$M$, with smooth charts adapted from the charts of $M$, for which the inclusion $\iota\colon A\hookrightarrow M$ is a smooth embedding.

Given a smooth fiber bundle $\pi\colon P\to B$, the corresponding {\it vertical bundle} is the bundle over the manifold $P$ with total space $\calv P:=\ker\d\pi$. The projection $\calv P\to P$ is the restriction of the tangent bundle projection $TP\to P$, and hence the vertical bundle is a subbundle of the tangent bundle. We will also make use of associated bundles. Given a Lie group~$K$, a~manifold~$P$ with a~free and proper right action of~$K$, and a manifold~$F$ with a proper left action of $K$, the {\it associated bundle} is the bundle over the smooth orbit space $P/K$ with total space $P\times^K F:=(P\times F)/K$. Here, the group $K$ acts on the space $P\times F$ by $k\cdot (p,f):= (p\cdot k^{-1},k\cdot f)$ in a free and proper fashion from the left. We will denote the elements of $P\times^KF$ by $[p,f]$. The bundle projection $P\times^KF\to P/K$ is defined by $[p,f]\mapsto K\cdot p$, where $K\cdot p$ is the $K$-orbit of~$p$. If the manifold $F$ is a product of the form $M\times N$, we will denote the elements of $P\times^KF$ by $[p,m,n]$ instead of~$[p,(m,n)]$.

\section{Relative equilibria and isomorphic vector fields}\label{IsosAndRelEqs}

In this section we define the groupoid of equivariant vector fields on a manifold with a group action, and the corresponding notion of isomorphism of equivariant vector fields. We then describe Krupa's construction in this language, and Lerman's results about the groupoids of equivariant vector fields present in this construction. Along the way, we discuss how relative equilibria are preserved by isomorphisms of equivariant vector fields, equivariant extension of vector fields, pushforward and pullbacks of vector fields (when these are defined), and certain functors between groupoids of equivariant vector fields.

We work in the following setting:

\begin{Definition}[$G$-manifold]\label{Def:GMan} A manifold $M$ with an action of a Lie group $G$ is called a {\it $G$-manifold}. If the action of $G$ is a proper action then we say $M$ is a {\it proper $G$-manifold}.
\end{Definition}

By an equivariant vector field we mean:

\begin{Definition}[equivariant vector field]\label{Def:InvVF}
A vector field $X$ on a manifold $M$ is {\it equivariant} with respect to the action of a Lie group $G$ if for all $g\in G$ we have $X\circ g_M=\d g_M\circ X$, where $g_M\colon M\to M$ is the diffeomorphism $m\mapsto g\cdot m$.
If we need to specify the group we say $X$ is $G$-equivariant.
\end{Definition}

We next recall the definition of a relative equilibrium:

\begin{Definition}[relative equilibrium]\label{Def:RelEq}
Given an equivariant vector field $X$ on a $G$-manifold~$M$, a point $m\in M$ is a {\it relative equilibrium} of $X$ if the vector $X(m)$ is tangent to the group orbit $G\cdot m$. If we need to specify the group we say $m$ is a $G$-relative equilibrium.
\end{Definition}

\begin{Definition}[velocities]\label{Def:Velocities}
Let $M$ be a proper $G$-manifold, let $X$ be an equivariant vector field on $M$, and let $m$ be a point in $M$.
A~{\it velocity} for the point $m$ is a vector $\xi\in\ffg$ such that $X(m)=\xi_M(m)$, where
\begin{gather*}
\xi_M\colon \ M\to TM, \qquad \xi_M(m):=\frac{\d}{\d t}\Big|_0\exp(t\xi)\cdot m
\end{gather*}
is the fundamental vector field generated by the vector $\xi$.
\end{Definition}

\begin{Remark}
Velocities exist for relative equilibria since
\begin{gather*}
T_m(G\cdot m)=\{\nu_M(m)\,|\, \nu\in\ffg\}.
\end{gather*}
In fact, the existence of velocities at a point characterize that point as a relative equilibrium. Furthermore, since
\begin{gather*}
\ffg_m=\{\eta\in\ffg\,|\,\eta_M(m)=0\},
\end{gather*}
velocities are unique modulo the Lie algebra $\ffg_m$ of the stabilizer.
\end{Remark}

The following maps are needed to define morphisms of vector fields:

\begin{Definition}[infinitesimal gauge transformations]\label{Def:InfGaugeTransf}
{\it Infinitesimal gauge transformations} are the elements of the vector space
\begin{gather*}
C^\infty(M,\ffg)^G:=\{\psi\colon M\to \ffg\,|\, \psi(g\cdot m)=\Ad(g)\psi(m) \ \text{for all} \ g\in G, \, m\in M\}.
\end{gather*}
\end{Definition}

\begin{Remark}\label{Rem:InfGaugeTransf}
If the action of $G$ on $M$ is free and proper, then the orbit space $M/G$ is a~mani\-fold and the orbit space map $M\to M/G$ is a principal $G$-bundle. In this case, the space of infinitesimal gauge transformations $C^\infty(M,\ffg)^G$ is isomorphic to the space of smooth sections of the bundle $M\times^G\ffg\to M/G$.
\end{Remark}

The space of infinitesimal gauge transformations $C^\infty(M,\ffg)^G$ acts, as a group under pointwise addition, on the space of equivariant vector fields $\Gamma(TM)^G$ by
\begin{gather*}
C^\infty(M,\ffg)^G \times \Gamma(TM)^G \to \Gamma(TM)^G, \qquad
(\psi,X) \mapsto X+\psi_M,
\end{gather*}
where $\psi_M$ denotes the vector field on $M$ defined by
\begin{gather*}
\psi_M\colon \ M\to TM, \qquad \psi_M(m):=\frac{\d}{\d t}\Big|_0\exp\left(t\psi(m)\right)\cdot m .
\end{gather*}

\begin{Lemma}\label{Lemma:InvPsiM}
Let $M$ be a $G$-manifold and let $\psi\colon M\to\ffg$ be an infinitesimal gauge transformation on $M$.
The induced vector field $\psi_M$ is an equivariant vector field with respect to the action of~$G$.
\end{Lemma}

\begin{proof}
This is a consequence of the naturality of the exponential. Let $g\in G$ and $m\in M$, then
\begin{align*}
\psi_M(g\cdot m)
&=\frac{\d}{\d t}\Big|_0 \exp(t\psi(g\cdot m))\cdot g\cdot m \\
&=\frac{\d}{\d t}\Big|_0 \exp(t\Ad(g)\psi(m))\cdot g\cdot m \qquad \text{ by the equivariance of }\psi \\
&=\frac{\d}{\d t}\Big|_0 g\exp(t\psi(m))g^{-1}\cdot g\cdot m \qquad \text{ by the naturality of }\exp \\
&=(\d g_M)_m \psi_M(m).
\end{align*}
Hence, $\psi_M$ is an equivariant vector field.
\end{proof}

We can now define the groupoid of equivariant vector fields:

\begin{Definition}[groupoid of equivariant vector fields]\label{Def:GpoidIVFs}
Let $M$ be a $G$-manifold.
The {\it groupoid of equivariant vector fields} $\bbx(G\times M\toto M)$ is the action groupoid $C^\infty(M,\ffg)^G\ltimes\Gamma(TM)^G$ corresponding to the action of the infinitesimal gauge transformations $C^\infty(M,\ffg)^G$ on the $G$-equivariant vector fields $\Gamma(TM)^G$.
The groupoid of $G$-equivariant vector fields has
\category
{ \text{equivariant vector fields } X\in \Gamma(TM)^G,}
{ \text{pairs }(\psi,X)\in C^\infty(M,\ffg)^G\times\Gamma(TM)^G.}
The source function is given by
\begin{gather*}
s\colon \ C^\infty(M,\ffg)^G\times\Gamma(TM)^G\to\Gamma(TM)^G, \qquad (\psi,X)\mapsto X,
\end{gather*}
and the target function is given by
\begin{gather*}
t\colon \ C^\infty(M,\ffg)^G\times\Gamma(TM)^G\to\Gamma(TM)^G, \qquad (\psi,X)\mapsto X+\psi_M.
\end{gather*}
The composition of a composable pair of morphisms $(\varphi,X+\psi_M)$ and $(\psi,X)$ is given by
\begin{gather*}
(\varphi,X+\psi_M) \circ (\psi,X)=(\varphi+\psi, X).
\end{gather*}
The unit function is given by
\begin{gather*}
u\colon \ \Gamma(TM)^G\to C^\infty(M,\ffg)^G\times\Gamma(TM)^G, \qquad X\mapsto (0,X),
\end{gather*}
and the inversion function is given by
\begin{gather*}
-\colon \ C^\infty(M,\ffg)^G\times\Gamma(TM)^G\to C^\infty(M,\ffg)^G\times\Gamma(TM)^G, \qquad (\psi,X)\mapsto (-\psi,X).
\end{gather*}
\end{Definition}

\begin{Remark}
Hepworth \cite{H09} defined vector fields on differentiable stacks and showed they form a category. In the case of a quotient stack $[M/G]$ for the action of a compact group $G$ on a~mani\-fold $M$, Hepworth showed that the category $\bbx\left([M/G]\right)$ of vector fields on the stack $[M/G]$ is equivalent to the category $\bbx(G\times M\toto M)$ given in Definition~\ref{Def:GpoidIVFs} \cite[Proposition~5.1]{H09}.
\end{Remark}

In the following definition we highlight what it means for two vector fields to be isomorphic in the groupoid $\bbx(G\times M\toto M)$ of equivariant vector fields:

\begin{Definition}[isomorphic vector fields]\label{Def:IsoVFs}
Two equivariant vector fields $X$ and $Y$ on a $G$-manifold $M$ are {\it $G$-isomorphic} if there exists an infinitesimal gauge transformation~$\psi$ in the space $C^\infty(M,\ffg)^G$ such that
\begin{gather*}
Y=X+\psi_M.
\end{gather*}
\end{Definition}

As noted in \cite[Corollary~2.8]{L15}, isomorphisms of equivariant vector fields preserve relative equilibria in the following sense:

\begin{Lemma}\label{Lemma:SameRelEqs}
Let $X$ and $Y$ be two isomorphic equivariant vector fields on a $G$-manifold $M$. If a point $m$ is a relative equilibrium of $X$ then it is a relative equilibrium of $Y$.
\end{Lemma}

\begin{proof}
Since $X$ and $Y$ are isomorphic, there exists a map $\psi\in C^\infty(M,\ffg)^G$ such that $Y=X+\psi_M$. Note that the vector $X(m)$ is tangent to the group orbit $G\cdot m$ since the point $m$ is a relative equlibrium of~$X$. The vector $\psi_M(m)$ is also tangent to the group orbit $G\cdot m$ since the vector $\psi_M(m)$ is defined to be the derivative of a curve on the group orbit of the point~$m$. Thus, we have
\begin{gather*}
Y(m)=X(m)+\psi_M(m)\in T_m(G\cdot m),
\end{gather*}
meaning the point $m$ is a relative equilibrium of the vector field $Y$.
\end{proof}

We will use the following vector fields in our application of Theorem \ref{Thm:MainTheorem1} to Hamiltonian relative equilibria in Section \ref{HamStability}:

\begin{Definition}[augmented vector fields]\label{Def:AugmentedVFs}
Let $M$ be a proper $G$-manifold and $X$ an equivariant vector field on $M$. Given a vector $\xi\in\ffg$, the corresponding {\it vector field augmented by} $\xi$ is the vector field
\begin{gather*}
X_\xi\colon \ M\to TM, \qquad X_\xi:=X-\xi_M.
\end{gather*}
\end{Definition}

\begin{Remark}
Given a $G$-equivariant vector field $X$ on a proper $G$-manifold $M$, the corresponding augmented vector field $X_\xi$ is not $G$-equivariant.
However, it is equivariant with respect to the Lie subgroup
\begin{gather*}
G_\xi:=\{g\in G \,|\, \Ad(g)\xi=\xi\}.
\end{gather*}
Also note, that if $\xi\in\ffg$ is a velocity for a $G$-relative equilibrium $m$ of the vector field $X$, then the augmented vector field $X_\xi$ has an equilibrium at the point $m$.
\end{Remark}

\begin{Lemma}\label{Lemma:AugIso}
Let $M$ be a proper $G$-manifold, let $X$ be an equivariant vector field on $M$, and let $\xi$ be a given vector in the Lie algebra $\ffg$ of $G$.
The vector field $X$ is $G_\xi$-isomorphic to its augmented vector field $X_\xi\in\Gamma(TM)^{G_\xi}$.
\end{Lemma}
\begin{proof}
Let $\ffg_\xi$ be the Lie algebra of the Lie subgroup $G_\xi$.
The constant map
\begin{gather*}
\xi\colon \ M\to\ffg_\xi, \qquad m\mapsto \xi
\end{gather*}
is a smooth $\Ad(G_\xi)$-equivariant map, and hence gives a morphism of the groupoid of $G_\xi$-equivariant vector fields $\bbx(G_\xi\times M\toto M)$. Note $X=X_\xi+\xi_M$ by definition, so the result follows.
\end{proof}

Recall we can assemble the maximal integral curves of a smooth vector field on a Hausdorff manifold into a maximal flow:

\begin{Definition}[Flow]\label{Def:Flow}
Let $M$ be a Hausdorff manifold and let $X$ be a smooth vector field on $M$.
For every point $m\in M$, let $\gamma_m\colon I_m\to M$ be the maximal integral curve of $X$ such that $\gamma_m(0)=m$.
Let $A$ be the open subset of $\bbr\times M$ defined by
\begin{gather*}
A:=\bigcup_{m\in M}I_m\times\{m\}.
\end{gather*}
The {\it maximal flow}, or just {\it flow}, of the vector field $X$ is the smooth map
\begin{gather*}
\phi\colon \ A\to M, \qquad \phi(t,m):=\gamma_m(t).
\end{gather*}
The set $A$ is called the {\it flow domain} of $\phi$.
\end{Definition}

\begin{Remark}It is important to recall that we are assuming all manifolds are Hausdorff, this is required for some of the definitions and results in this paper. From now on, we won't explicitly mention this hypothesis.
\end{Remark}

The following result, due to Lerman, relates the flows of isomorphic vector fields:

\begin{Theorem}[{Lerman \cite[Theorem~1.6]{L15}}]\label{Thm:IsoFlows}
Let $M$ be a proper $G$-manifold and let $X$ and $Y$ be two isomorphic equivariant vector fields on $M$. Then there exists a family of smooth maps $\{F_t\colon M\to G\}$ depending smoothly on $t$ so that the maximal flows $\phi^X$ and $\phi^Y$, of $X$ and $Y$ respectively, satisfy
\begin{gather*}
\phi^X(t,m)=F_t(m)\cdot \phi^Y(t,m)
\end{gather*}
for all $(t,m)\in\bbr\times M$ in the domain of the flow $\phi^X$.
\end{Theorem}

We recall the following notion of continuous flows on topological spaces:

\begin{Definition}\label{Def:AbsFlow}
Let $Z$ be a topological space, and let $B$ be an open subset of $\bbr\times Z$ containing the set $\{0\}\times Z$. An {\it abstract flow} on $Z$ is a continuous map $\Phi\colon B\to Z$ satisfying:
\begin{enumerate}\itemsep=0pt
\item[1)] $\Phi(0,z)=z$ for all $z\in Z$;
\item[2)] $\Phi(t,\Phi(s,z))=\Phi(s+t,z)$ whenever both sides make sense.
\end{enumerate}
For a given point $z\in Z$, the {\it curve of the abstract flow starting at $z$} is the curve $\gamma_z\colon I_z\to Z$ defined by $\gamma_z(t)=\Phi(t,z)$, where $I_z$ consists of all times $t$ for which $(t,z)\in B$.
\end{Definition}

We recall the following standard result about $G$-equivariant vector fields:

\begin{Lemma}\label{Lemma:InducedOrbitFlow}
Let $M$ be a proper $G$-manifold and let $X$ be an equivariant vector field on $M$. The flow $\phi\colon A\to M$ of the vector field $X$ induces an abstract flow $\Phi\colon B\to M/G$ on the orbit space $M/G$ such that the following diagram commutes:
\begin{gather}\label{Diag:OrbitFlow}
\xy
(-14,10)*+{A}="1";
(12,10)*+{M}="2";
(-14,-6)*{B}="3";
(12,-6)*+{M/G,}="4";
{\ar@{->}^{\phi} (-12,10);(8,10)};
{\ar@{->}_{\id\times\pi} (-14,7.5);(-14,-3)};
{\ar@{->}^{\pi} "2";"4"};
{\ar@{->}_{\Phi} "3";"4"};
\endxy
\end{gather}
where $\pi\colon M\to M/G$ is the orbit map.
\end{Lemma}

\begin{proof}
First, define the set $B:=(\id\times\pi)(A)\subseteq \bbr\times M/G$ and the map
\begin{gather*}
\Phi\colon\ B\to M/G, \qquad \Phi(t,\pi(m)):=\pi(\phi(t,m)).
\end{gather*}
We want to show that the map $\Phi$ is our desired abstract flow. Thus, we need to show that the set $B$ is open, that $B$ contains $\{0\}\times M/G$, that the map $\Phi$ is well-defined and continuous, that~$\Phi$ makes the diagram~(\ref{Diag:OrbitFlow}) commute, and that $\Phi$ satisfies properties~(1) and~(2) in Definition~\ref{Def:AbsFlow}.

Observe that the action of the Lie group $G$ on the manifold $M$ gives an action on the product $\bbr\times M$ by
\begin{gather*}
g\cdot (t,m) := (t,g\cdot m)
\end{gather*}
for all $g\in G$ and $(t,m)\in\bbr\times M$. The orbit space of this action is the product $\bbr\times M/G$ and the quotient map is $\id\times\pi\colon \bbr\times M\to\bbr\times M/G$, where $\pi\colon M\to M/G$ is the quotient map of the given action. To see that the set $B$ is open, it suffices to check that the open set $A$ is saturated with respect to the quotient map $\id\times\pi$, or equivalently that it is $G$-invariant with respect to the action of $G$ on the product $\bbr\times M$. For this, let $(t,m)\in A$ and note, using the equivariance of the vector field $X$, that the curve given by $t\mapsto g\cdot \phi(t,m)$ is the maximal integral curve of~$X$ starting at~$g\cdot m$. In particular, it is defined for the same times $t$ that the integral curve starting at the point $m$ is defined. Thus, if $(t,m)\in A$ then $(t,g\cdot m)\in A$, or equivalently the flow domain~$A$ is $G$-invariant. Furthermore, note that $(\id\times\pi)(\{0\}\times M)=\{0\}\times M/G$.
Hence,
\begin{gather*}
\{0\}\times M/G=(\id\times\pi)(\{0\}\times M)\subseteq (\id\times\pi)(A)=B
\end{gather*}
as desired.

Next, note that the map $\Phi$ is well-defined by the equivariance of the flow $\phi$. Furthermore, the map $\Phi$ makes the square~(\ref{Diag:OrbitFlow}) commute by definition. By the characteristic property of the quotient topology, the map $\Phi$ is continuous if and only if the map $\Phi\circ (\text{id}\times \phi)$ is continuous. Since the diagram (\ref{Diag:OrbitFlow}) commutes, we have that $\Phi\circ (\text{id}\times \phi)=\pi\circ\phi$. Since $\pi\circ\phi$ is the composition of continuous maps, then the map $\Phi$ is continuous.

Now, for every point $\pi(m)\in M/G$ we have
\begin{gather*}
\Phi(0,\pi(m))=\pi(\phi(0,m))=\pi(m)
\end{gather*}
since $\phi(0,m)=m$. Similarly, the second property follows by the corresponding property of the flow $\phi$
\begin{gather*}
\Phi\big(t,\Phi(s,\pi(m))\big)
=\Phi\big(t,\pi(\phi(s,m))\big)
=\pi\big(\phi(t,\phi(s,m))\big)\\
\hphantom{\Phi\big(t,\Phi(s,\pi(m))\big)}{}
=\pi\big(\phi(s+t,m)\big)
=\Phi\big(s+t,\pi(m)\big).
\end{gather*}
Hence, the map $\Phi\colon B\to M/G$ is the desired abstract flow.
\end{proof}

The following is a corollary of Theorem \ref{Thm:IsoFlows} and Lemma \ref{Lemma:InducedOrbitFlow} (also see \cite[Corollary~2.8]{L15}):

\begin{Corollary}\label{Cor:SameOrbitFlow}
Let $M$ be a proper $G$-manifold, and let $X$ and $Y$ be two isomorphic equivariant vector fields on $M$.
Then the maximal flows of $X$ and $Y$ have the same domain and induce the same abstract flow on the orbit space $M/G$.
\end{Corollary}

\begin{proof}
Let $\phi^X$ an $\phi^Y$ be the maximal flows of the vector fields $X$ and $Y$ respectively.
Let $\{F_t\colon M\to G\}$ be the family of maps relating the flows (see Theorem \ref{Thm:IsoFlows}).
Thus, for all pairs $(t,m)$ in the domain of the flow $\phi^Y$ we have
\begin{gather}\label{Eq:IsoFlows}
\phi^X(t,m)=F_t(m)\cdot \phi^Y(t,m).
\end{gather}
In particular, note that any pair $(t,m)$ in the domain of the flow $\phi^Y$ is in the domain of the flow $\phi^X$.
Reversing the role of $X$ and $Y$ in Theorem \ref{Thm:IsoFlows} gives the opposite inclusion of the flow domains.
Hence, the flows $\phi^X$ and $\phi^Y$ have the same domain.

Now let $\Phi^X\colon B\to M/G$ and $\Phi^Y\colon B\to M/G$ be the induced flows on the orbit space of $X$ and $Y$ respectively.
Using equality (\ref{Eq:IsoFlows}) and the definition of the induced orbit space flow given in the proof of Lemma \ref{Lemma:InducedOrbitFlow}, we have that
\begin{gather*}
\Phi^X\big(t,\pi(m)\big)
=\pi\big(\phi^X(t,m)\big)
=\pi\big(F_t(m)\cdot \phi^Y(t,m)\big)
=\pi\big(\phi^Y(t,m)\big)
=\Phi^Y\big(t,\pi(m)\big).
\end{gather*}
Hence, the induced flows on the orbit space are equal.
\end{proof}

Abstract flows can also have fixed points:

\begin{Definition}[fixed point of an abstract flow]\label{Def:AbsFlowFixedPoint}
Let $\phi\colon A\to Z$ be an abstract flow on a~topological space $Z$. A {\it fixed point} of the flow $\phi$ is a~point $z\in Z$ such that $\phi(t,z)=z$ for all times $t$ with $(t,z)\in A$.
\end{Definition}

\begin{Remark}
Let $M$ be a proper $G$-manifold and let $\pi\colon M\to M/G$ be the quotient map. Observe that if $X$ is an equivariant vector field on $M$, then a~point $m\in M$ is a relative equilibrium of $X$ if and only if the point $\pi(m)\in M/G$ is a fixed point of the induced abstract flow on the orbit space.
\end{Remark}

We proceed to describe Krupa's decomposition following Lerman~\cite{L15}. We begin by recalling saturations and equivariant extension:

\begin{Definition}[saturation]\label{Def:Saturation}
Given a $G$-manifold $M$ and a subset $A\subseteq M$, the {\it saturation} of~$A$ is the subset of $M$ defined by $G\cdot A:=\{g\cdot a \,|\, g\in G,\, a\in A\}$.
\end{Definition}

Recall that equivariant maps out of regular submanifolds of a proper $G$-manifold have unique equivariant extensions to the saturation of the submanifold, provided some additional hypotheses as in the following standard lemma (see also \cite[Lemma~2.10.1]{F07}):

\begin{Lemma}\label{Lemma:EquivExtension}
Let $M$ and $N$ be proper $G$-manifolds, let $K$ be a Lie subgroup of $G$, let $A$ be a~$K$-invariant regular submanifold of $M$, and let $f\colon A\to N$ be a $K$-equivariant map. Suppose that the map $G\times A\to G\cdot A$ given by $(g,a)\mapsto g\cdot a$ descends to a diffeomorphism from the associated bundle $G\times^KA$ to the saturation $G\cdot A$. Then there exists a unique $G$-equivariant extension of the map $f$ given by
\begin{gather*}
\varepsilon f\colon \ G\cdot A\to N, \qquad \varepsilon f (g\cdot a):=g\cdot f(a).
\end{gather*}
\end{Lemma}

\begin{proof} Define the $G$-equivariant map
\begin{gather*}
G\times A \to N, \qquad (g,a)\mapsto g\cdot f(a).
\end{gather*}
By using the $K$-equivariance of the map $f$, note that this map is $K$-invariant with respect to the action of $K$ on $G\times A$. Thus, this map descends to a smooth $G$-equivariant map
\begin{gather*}
G\times ^K A \to N, \qquad [g,a]\mapsto g\cdot f(a).
\end{gather*}
Using the diffeomorphism between the associated bundle $G\times ^KA$ and the saturation $G\cdot A$, we obtain the smooth extension $\varepsilon f$.
To see it is unique, suppose that $F\colon G\cdot A\to N$ is any other $G$-equivariant extension.
Then note $F(g\cdot a)=g\cdot F(a)=g\cdot f(a) = \varepsilon f(g\cdot a)$, and hence $F=\varepsilon f$.
\end{proof}

We recall the definition of a slice:

\begin{Definition}[slices]\label{Def:Slices}
Given a $G$-manifold $M$, let $G_m$ be the stabilizer of a point $m\in M$.
A~{\it slice} for the action through $m$ is a $G_m$-manifold $V$ and a $G_m$-equivariant embedding $j\colon V\to M$ such that
\begin{enumerate}\itemsep=0pt
\item[1)] the point $m$ is in the image $j(V)$;
\item[2)] the saturation $G\cdot j(V)$ is open in $M$;
\item[3)] the map
\begin{gather*}
G\times V \to G\cdot j(V), \qquad (g,v)\mapsto g\cdot j(v)
\end{gather*}
descends to a $G$-equivariant diffeomorphism
\begin{gather*}
G\times^{G_m}V\to G\cdot j(V), \qquad [g,v]\mapsto g\cdot j(v),
\end{gather*}
where $G\times^{G_m}V:=(G\times V)/G_m$ is the associated bundle.
\end{enumerate}
For the sake of conciseness, we often write $G\cdot V$ instead of $G\cdot j(V)$.
\end{Definition}

\begin{Remark}\label{Rem:SliceExistence}
It is a classic theorem of Palais \cite{Pa61} that slices exist for points in \textit{proper} $G$-manifolds (see also \cite[Theorem~2.3.3]{DK00}).
In proper $G$-manifolds, it is also possible and convenient to take the slice $V$ through a point $m$ to be an open ball around the origin of a vector space with a representation of the stabilizer $G_m$ (see, for example, \cite[Theorem~B.24]{GuK02}).
\end{Remark}

The following definition will be convenient for the sake of brevity:

\begin{Definition}[proper $G$-manifold with slice]\label{Def:PGManWSlice}
A {\it proper $G$-manifold with slice} is a quintuple $(M,G,m,V,j)$ consisting of a proper $G$-manifold $M$, a point $m$ on the manifold, and a~slice~$V$ for the action through the point $m$ with corresponding $G_m$-equivariant embedding $j\colon V\to M$.
\end{Definition}

\begin{Remark}\label{Rem:SliceFacts}
Let $(M,G,m,V,j)$ be a proper $G$-manifold with slice. The following facts will be important:
\begin{enumerate}\itemsep=0pt
\item The bundle $G\cdot V\to G\cdot m$ has typical fiber $V$ and is $G$-equivariantly diffeomorphic to the associated bundle $G\times^{G_m}V\to G/G_m$ (see, for example, \cite[Theorem~2.4.1]{DK00}).
In other words, the following diagram, with the canonical maps, commutes:
\begin{gather*}
\xy
(-14,10)*+{G\times^{G_m}V}="1";
(12,10)*+{G\cdot V}="2";
(-14,-6)*{G/G_m}="3";
(12,-6)*+{G\cdot m.}="4";
{\ar@{->}^{\cong} "1";"2"};
{\ar@{->} "1";"3"};
{\ar@{->} "2";"4"};
{\ar@{->}^{\cong} "3";"4"};
\endxy
\end{gather*}
We think of $G\cdot V$ as a tubular neighborhood of the group orbit $G\cdot m$ and often refer to it as a tube. Thus, the associated bundle $G\times^{G_m}V\to G/G_m$ serves as a model for the tubular neighborhood $G\cdot V$, and we will sometimes identify $G\times^{G_m}V$ and $G\cdot V$.
\item Definition \ref{Def:Slices} implies that the tangent space at the point $m$ splits in the form
\begin{gather*}
T_mM=T_m(G\cdot m)\oplus T_mj(V),
\end{gather*}
while for any point $v\in V$ we have
\begin{gather*}
T_{j(v)}M=T_{j(v)}(G\cdot j(v)) + T_{j(v)}j(V).
\end{gather*}
Definition \ref{Def:Slices} also implies that for any point $v\in V$, if a group element $g\in G$ is such that $g\cdot j(v)\in j(V)$ then $g\in G_m$.
\end{enumerate}
\end{Remark}

\begin{Remark}\label{Rem:ArbitraryModel}
By the previous remark, we can model tubes generated by slices by considering arbitrary associated bundles of the form $G\times ^KV$, where $K$ is a compact Lie subgroup of a Lie group $G$, and $V$ is an open ball around the origin in a vector space with a representation of~$K$.
For such models, note that the point $m:=[1,0]$ has as stabilizer the Lie subgroup $K$ acting on $G\times ^KV$ as a subgroup of $G$.
Therefore, the $K$-manifold $V$ with the $K$-equivariant embedding $j\colon V\hookrightarrow G\times^KV$ defined by $j(v):=[1,v]$, is a slice for the action through the point $m$.
\end{Remark}

\begin{Remark}\label{Rem:Groupoids}
A proper $G$-manifold with slice $(M,G,m,V,j)$ gives rise to two action groupoids, namely:
\begin{itemize}\itemsep=0pt
\item the action groupoid $G_m\times V\toto V$ of the slice;
\item the action groupoid $G\times(G\cdot V)\toto G\cdot V$ of the tube.
\end{itemize}
Thus, the choice of slice gives rise to two groupoids of equivariant vector fields in the sense of Definition \ref{Def:GpoidIVFs}:
\begin{itemize}\itemsep=0pt
\item the groupoid $\bbx(G_m\times V \toto V)$ of $G_m$-equivariant vector fields on the slice $V$;
\item the groupoid $\bbx(G\times G\cdot V \toto G\cdot V)$ of $G$-equivariant vector fields on the tube $G\cdot V$.
\end{itemize}
It is a theorem of Lerman that these groupoids are equivalent (see \cite[Theorem~1.16]{L15}).
This theorem was stated using $2$-term chains of topological vector spaces. In Theorem~\ref{Thm:LermanEquivalence} we state his result using an equivalent formulation.
\end{Remark}

Given a proper $G$-manifold with slice, we can use the embedding of the slice to push forward vector fields and infinitesimal gauge transformations onto the image of the slice. We can then extend these uniquely to the tube as in Lemma~\ref{Lemma:EquivExtension}. This assembles into a canonical functor as follows:

\begin{Definition}[equivariant extension functor]\label{Def:ExtensionFunctor}
Let $(M,G,m,V,j)$ be a proper $G$-manifold with slice. The {\it equivariant extension functor} is the functor
\begin{gather*}
\xy
(-25,10)*+{E\colon \ \bbx(G_m\times V\toto V)}="1";
(25,10)*+{\bbx(G\times G\cdot V\toto G\cdot V),}="2";
{\ar@{->} "1";"2"};
(-40,0)*+{\big(X}="3";
(-20,0)*+{Y\big)}="4";
{\ar@{->}^{(\psi,X)} "3";"4"};
(5,0)*+{\big(E(X)}="5";
(40,0)*+{E(Y)\big),}="6";
{\ar@{->}^{(E(\psi),E(X))} "5";"6"};
{\ar@{|->} (-10,0);(-5,0)};
\endxy
\end{gather*}
where for any vector field $X\in\Gamma(TV)^{G_m}$ we define
\begin{gather*}
E(X)\colon \ G\cdot V\to T(G\cdot V), \qquad E(X)\big(g\cdot j(v)\big):=\d (g\circ j)X(v),
\end{gather*}
and for any infinitesimal gauge transformation $\psi\in C^\infty(V,\ffg_m)^{G_m}$ we define
\begin{gather*}
E(\psi)\colon \ G\cdot V\to\ffg, \qquad E(\psi)(g\cdot j(v)):=\Ad(g)\psi(v).
\end{gather*}
\end{Definition}

\begin{Remark} The equivariant extension functor makes use of push-forwards by the slice embedding and of equivariant extension as in Lemma~\ref{Lemma:EquivExtension} at both the object and morphism level. Let $(M,G,m,V,j)$ be a proper $G$-manifold with slice. Using the notation of Lemma \ref{Lemma:EquivExtension}, the equivariant extension functor $E$ satisfies
\begin{gather*}
E(X)=\varepsilon (j_* X) \qquad\text{and}\qquad E(\psi)=\varepsilon (j_* \psi)
\end{gather*}
for any equivariant vector field $X$ and any infinitesimal gauge transformation $\psi$ on the slice. Furthermore, note that the image under the functor $E$ of the space of $G_m$-equivariant vector fields on the slice consists of the space of $G$-equivariant vertical vector fields on the bundle $G\cdot V\to G\cdot m$. That the image is contained in the space of $G$-equivariant vertical vector fields follows from the definition.
That the functor on objects is surjective onto the vertical vector fields can be shown by using the functor of Definition~\ref{Def:ProjectionFunctor}.
\end{Remark}

The functor $E$ is only part of the equivalence stated in Remark~\ref{Rem:Groupoids}. For a functor in the opposite direction we first need to obtain a connection via a choice of Lie algebra splitting, as follows:

\begin{Lemma}\label{Lemma:SplitConn}
Let $G$ be a Lie group with Lie algebra $\ffg$, let $K$ be a Lie subgroup with Lie algebra~$\ffk$, and let $A$ be a proper $K$-manifold.
Then a choice of $K$-equivariant splitting $\ffg=\ffk\oplus\ffq$ gives rise to a $G$-equivariant connection on the associated bundle $G\times^K A\to G/K$.
\end{Lemma}

\begin{proof}
We show that the given splitting of the Lie algebra gives rise to a bundle projection from the tangent bundle $T(G\times^KA)=TG\times^KTA$ to the vertical bundle $\calv(G\times^KA)=G\times^KTA$. Here, recall that the vertical bundle $\calv(G\times^KA)$ is a bundle over the total space $G\times ^K A$ of the associated bundle $G\times^KA \to G/K$. Thus, we show that the Lie algebra splitting induces a~connection $\widetilde \Phi \in \Omega^1\big(G\times^KA; \calv(G\times^KA)\big)$.

The Lie algebra splitting gives rise to a $K$-equivariant projection $\bbp\colon \ffg\to\ffk$. The projection~$\bbp$ in turn gives rise to a principal connection $\Phi\in\Omega^1 (G ; \calv(G) )$ on the principal $K$-bundle $G\to G/K$, where the subgroup $K$ acts on $G$ by right-multiplication. For any $g\in G$ and $X\in T_gG$, this principal connection is given by
\begin{gather*}
\Phi_g(X):=(\d L_g)_1\bbp \big(\big(\d L_{g^{-1}}\big)_g (X)\big).
\end{gather*}
The remaining part of the argument consists of showing that the principal connection $\Phi$ induces a connection on the associated bundle $G\times^KA\to G/K$.
This part of the argument is standard.
However, we include an overview here so that we can refer to the construction in the sequel (for more details see, for example, \cite[Section~11.8]{KMS93}).

Consider the quotient map $\varpi\colon TG\times TA \to TG\times ^KTA$ and the $G$-equivariant map
\begin{gather*}
\Phi\times \id\colon \ TG\times TA \to TG \times TA, \qquad (\Phi\times \id)(X,Y)= (\Phi(X),Y ).
\end{gather*}
Since the composition $\varpi\circ (\Phi\times\text{id})$ is $K$-invariant with respect to the action of the subgroup $K$ on the product $TG\times TA$, there exists a unique smooth map $\widetilde \Phi$ such that the following diagram commutes:
\begin{gather}\label{Diag:Connection}
\xy
(-20,10)*+{TG\times TA}="1";
(15,10)*+{TG\times TA}="2";
(-20,-6)*+{TG\times^KTA}="3";
(15,-6)*+{TG\times^KTA.}="4";
{\ar@{->}^{\Phi\times\text{id}} "1";"2"};
{\ar@{->}^{\varpi} "2";"4"};
{\ar@{->}_{\varpi} "1";"3"};
{\ar@{->}_{\widetilde\Phi} "3";"4"};
\endxy
\end{gather}
The map $\widetilde \Phi$ is idempotent since the map $\Phi$ is idempotent. Also, the image of the map $\widetilde \Phi$ is the vertical bundle $\calv(G\times^KA)=G\times^KTA$. Hence, $\widetilde\Phi$ is a projection, so it gives a connection on the associated bundle $G\times^{K}A\to G/K$. Furthermore, the map $\widetilde\Phi$ is $G$-equivariant since the map~$\Phi\times \text{id}$ is $G$-equivariant and the $G$-action commutes with the quotient map~$\varpi$. Hence, the map $\widetilde \Phi$ gives the desired $G$-equivariant connection.
\end{proof}

\begin{Definition}[connection induced by a splitting]\label{Def:SplitConnProj}
Let $(M,G,m,V,j)$ be a proper $G$-manifold with slice and let $\ffg=\ffg_m\oplus \ffq$ be a $G_m$-equivariant splitting.
The {\it connection induced by the splitting} is the $G$-equivariant connection $\widetilde \Phi \in \Omega^1\big(G\cdot V;\calv(G\cdot V)\big)$ obtained from Lemma~\ref{Lemma:SplitConn} by setting $K=G_m$, setting $A=V$, and using the canonical $G$-equivariant diffeomorphism $G\times^{G_m}V \cong G\cdot V$. The {\it vertical projection of vector fields induced by the splitting} is the map:
\begin{gather*}
\nu\colon \ \Gamma(T(G\cdot V))\to \Gamma(\calv(G\cdot V)), \qquad \nu(X):=\widetilde \Phi \circ X.
\end{gather*}
\end{Definition}

\begin{Remark}\label{Rem:VertProjIsEquiv}
Since the connection of Definition \ref{Def:SplitConnProj} is equivariant, the vertical projection of vector fields maps $G$-equivariant vector fields to $G$-equivariant vector fields. Hence, we may also take the vertical projection $\nu$ to be a map $\Gamma(T(G\cdot V))^G\to \Gamma(\calv(G\cdot V))^G$. In fact, if $H$ is any Lie subgroup of $G$, the vertical projection takes $H$-equivariant vector fields to $H$-equivariant vector fields. Hence, we may also take the vertical projection of vector fields to be a map $\Gamma(T(G\cdot V))^H\to \Gamma(\calv(G\cdot V))^H$.
\end{Remark}

Thus, we obtain the following functor that generalizes Krupa's decomposition from \cite{K90}:

\begin{Definition}[projection functor]\label{Def:ProjectionFunctor}
Let $(M,G,m,V,j)$ be a proper $G$-manifold with slice, let $\ffg=\ffg_m\oplus\ffq$ be a $G_m$-equivariant splitting, let $\bbp\colon \ffg\to\ffg_m$ be the corresponding $G_m$-equivariant projection, and let $\nu\colon \Gamma(T(G\cdot V))^G\to\Gamma(\calv(G\cdot V))^G$ be the vertical projection of equivariant vector fields induced by the splitting (see Definition~\ref{Def:SplitConnProj}). The {\it projection functor} corresponding to the Lie algebra splitting $\ffg=\ffg_m\oplus\ffq$ is the functor
\begin{gather*}
\xy
(-25,10)*+{P\colon \ \bbx(G\times G\cdot V\toto G\cdot V)}="1";
(25,10)*+{\bbx(G_m\times V\toto V),}="2";
{\ar@{->} "1";"2"};
(-45,0)*+{\big(X}="3";
(-25,0)*+{Y\big)}="4";
{\ar@{->}^{(\psi,X)} "3";"4"};
(0,0)*+{\big(P(X)}="5";
(35,0)*+{P(Y)\big),}="6";
{\ar@{->}^{(P(\psi),P(X))} "5";"6"};
{\ar@{|->} (-15,0);(-10,0)};
\endxy
\end{gather*}
where for any vector field $X\in\Gamma(T(G\cdot V))^G$ we define
\begin{gather*}
P(X)\colon \ V\to TV, \qquad P(X):= j^*(\nu(X)),
\end{gather*}
and for any map $\psi\in C^\infty(G\cdot V,\ffg)^G$ we define
\begin{gather*}
P(\psi)\colon \ V\to\ffg_m, \qquad P(\psi):=j^*(\bbp\circ\psi).
\end{gather*}
\end{Definition}

\begin{Remark}\label{Rem:Pullbacks}
Recall that given a smooth embedding we can pull back those vector fields on the target manifold that are tangent to the image of the embedding.
Hence, if $(M,G,m,V,j)$ is a~proper $G$-manifold; the vertical vector fields on the bundle $G\cdot V\to G\cdot m$ can be pulled-back by the embedding~$j$. Consequently, the functor~$P$ of Definition~\ref{Def:ProjectionFunctor} is well-defined on objects.
\end{Remark}

We can now state the equivalence of groupoids mentioned in Remark~\ref{Rem:Groupoids}, which is due to Lerman. Instead of the functors of Definitions~\ref{Def:ExtensionFunctor} and~\ref{Def:ProjectionFunctor}, Lerman used an equivalent formulation in terms of $2$-term chain complexes.
We state the equivalence using the functor formulation:

\begin{Theorem}[{Lerman, \cite[Theorem~4.3]{L15}}]\label{Thm:LermanEquivalence}
Let $(M,G,m,V,j)$ be a proper $G$-manifold with slice. The equivariant extension functor $E\colon \bbx(G_m\times V\toto V)\to \bbx(G\times G\cdot V\toto G\cdot V)$ $($see Definition~{\rm \ref{Def:ExtensionFunctor})} and the projection functor $P\colon \bbx(G\times G\cdot V\toto G\cdot V)\to \bbx(G_m\times V\toto V)$ corresponding to a choice of $G_m$-equivariant splitting $\ffg=\ffg_m\oplus\ffq$ $($see Definition~{\rm \ref{Def:ProjectionFunctor})} form an equivalence of categories. In particular, such functors satisfy
\begin{gather*}
P\circ E = \id \qquad\text{and}\qquad E\circ P \simeq \id.
\end{gather*}
For a given equivariant vector field $X$ on the tube $G\cdot V$, the natural isomorphism $\alpha\colon E\circ P \Rightarrow \id$ is of the form
\begin{gather*}
\alpha_X=\big(\psi^X,E(P(X))\big),
\end{gather*}
where $\psi^X\in C^\infty(G\cdot V, \ffg)^G$ is an infinitesimal gauge transformation taking values in the complement $\ffq$.
Thus, the map $\psi^X$ is such that
\begin{gather*}
X=E(P(X))+\psi^X_{G\cdot V},
\end{gather*}
where $\psi^X_{G_\cdot V}$ is the vector field induced by the map $\psi^X$.
\end{Theorem}

\begin{Remark}\label{Rem:SliceProjIndep}
Lerman introduced this approach to Krupa's decomposition to quantify the result of the choices in slice and projection. The choice in slice is adressed as follows. Let $M$ be a proper $G$-manifold and $m$ a point in $M$. If $V_1$ and $V_2$ are two slices for the action through the point $m$, then the corresponding groupoids $\bbx(G_m\times V_1\toto V_1)$ and $\bbx(G_m\times V_2\toto V_2)$ of $G_m$-equivariant vector fields are isomorphic groupoids \cite[Lemma~3.21]{L15}. After perhaps shrinking the slices, the isomorphism is induced by a $G_m$-equivariant diffeomorphism between the slices. The choice in projection, or equivalently the choice of Lie algebra splitting, is addressed as follows. Given a proper $G$-manifold with slice $(M,G,m,V,j)$, and two choices of $G_m$-equivariant splittings
\begin{gather*}
\ffg=\ffg_m\oplus\ffq_1=\ffg_m\oplus\ffq_2,
\end{gather*}
the corresponding projection functors
\begin{gather*}
P_1,P_2\colon \ \bbx(G\times G\cdot V\toto G\cdot V)\to \bbx(G_m\times V\toto G\times V)
\end{gather*}
are naturally isomorphic \cite[Lemma~3.17]{L15}.
\end{Remark}

As may be expected, the functors we have introduced preserve relative equilibria. We prepare for the proof of this fact via Lemmas \ref{Lemma:PBRelEq}, \ref{Lemma:ExtRelEq}, and \ref{Lemma:VertRelEq}.

\begin{Lemma}\label{Lemma:PBRelEq}
Let $M$ and $N$ be proper $G$-manifolds and let $f\colon M\to N$ be a $G$-equivariant diffeomorphism.
Suppose that $X$ and $Y$ are $f$-related equivariant vector fields on $M$ and $N$ respectively.
Then a point $m\in M$ is a relative equilibrium of the vector field $X$ if and only if the point $f(m)$ is a relative equilibrium of the vector field $Y$.
Thus, pullbacks and pushforwards of vector fields by equivariant diffeomorphisms preserve relative equilibria.
\end{Lemma}

\begin{proof}
The verification is a straightforward computation using the equation $\d f\circ X = Y\circ f$.
First, suppose $m$ is a $G$-relative equilibrium of the vector field $X$. Then
\begin{gather*}
Y(f(m)) =(\d f)_m(X(m))\in (\d f)_m(T_m(G\cdot m))=T_{f(m)}(G\cdot f(m)),
\end{gather*}
where $(\d f)_m(T_m(G\cdot m))=T_{f(m)}(G\cdot f(m))$ follows by the equivariance of the diffeomorphism~$f$. Thus, the point $f(m)$ is a $G$-relative equilibrium of the vector field $Y$. The converse is completely analogous.
\end{proof}

\begin{Lemma}\label{Lemma:ExtRelEq}
Let $M$ be a proper $G$-manifold, let $K$ be a Lie subgroup of~$G$, and let~$A$ be a~$K$-invariant regular submanifold of $M$ satisfying the hypotheses of Lemma~{\rm \ref{Lemma:EquivExtension}}. Suppose that~$X$ is a~$K$-equivariant vector field on $A$ and that the point $a\in A$ is a $K$-relative equilibrium of the vector field $X$. Then the point~$a$ is a~$G$-relative equilibrium of the equivariant extension~$\varepsilon X$ of~$X$. That is, equivariant extension preserves relative equilibria.
\end{Lemma}

\begin{proof}
This is essentially a corollary of Lemma~\ref{Lemma:PBRelEq}. Let $\iota\colon A\hookrightarrow G\cdot A$ be the inclusion of the submanifold $A$.
Note that the tube $G\cdot A$ is a $K$-manifold and that the inclusion $\iota$ is a~$K$-equivariant diffeomorphism onto its image $A$. Observe that the vector fields $X$ and $\varepsilon X$ are $\iota$-related; in fact, $\varepsilon X$ restricts to $X$ on $A$. Thus, by Lemma~\ref{Lemma:PBRelEq}, we know that $a$ is a $K$-relative equilibrium of the vector field $\varepsilon X$; that is, $\varepsilon X(a)\in T_a(K\cdot a)$. Since $A$ is a~regular submanifold, the tangent space $T_a(K\cdot a)$ is contained in the tangent space $T_a(G\cdot a)$. Hence, the point~$a$ is a~$G$-relative equilibrium of~$\varepsilon X$.
\end{proof}

\begin{Lemma}\label{Lemma:VertRelEq}
Let $G$ be a Lie group, let $K$ be a Lie subgroup, let $A$ be a proper $K$-manifold, and let $\ffg=\ffk\oplus\ffq$ be a $K$-equivariant splitting.
Let $\nu\colon \Gamma(T(G\times^{K} A))^G\to\Gamma(\calv(G\times^{K} A))^G$ be the vertical projection induced by the splitting $($Definition~{\rm \ref{Def:SplitConnProj})}, and let~$X$ be a $G$-equivariant vector field on the associated bundle $G\times^KA$. Then if the point $m\in G\times^KA$ is a $G$-relative equilibrium of the vector field~$X$, it is also a $G$-relative equilibrium of the vertical projection $\nu(X)$. That is, the vertical projection $\nu$ preserves relative equilibria.
\end{Lemma}

\begin{proof}
Consider the quotient maps:
\begin{gather*}
\pi\colon \ G\times A \to G\times^KA, \qquad
\varpi\colon \ TG\times TA\to TG\times^KTA.
\end{gather*}
Let $X$ be an equivariant vector field on the associated bundle $G\times^KA$ and suppose that the point $p=(g,a)\in G\times A$ is such that the point $m:=\pi(p)\in G\times^KA$ is a relative equilibrium of the vector field~$X$. Let $\Phi\in \Omega^1(G ; \calv G )$ and $\widetilde \Phi \in \Omega^1\big(G\times^KA;\calv(G\times^KA)\big)$ be the connections induced by the splitting of the Lie algebra (see Definition~\ref{Def:SplitConnProj} and the proof of Lemma~\ref{Lemma:SplitConn}), and let the map
\begin{gather*}
\nu\colon \ \Gamma\big(T\big(G\times^KA\big)\big)^G\to\Gamma\big(\calv\big(G\times^KA\big)\big)^G
\end{gather*}
be the vertical projection of equivariant vector fields with respect to this connection (Definition~\ref{Def:SplitConnProj}).
We want to show that the point $m$ is a $G$-relative equilibrium of the vector field~$\nu(X)$; that is, we want to show that
\begin{gather*}
\nu(X)(m)\in T_m(G\cdot m).
\end{gather*}

Since the action of $G$ commutes with the quotient maps $\pi$ and $\varpi$, observe that
\begin{gather}\label{Eq:Conn1}
\varpi (T_p(G\cdot p) ) =T_{m}(G\cdot m).
\end{gather}
Furthermore, using that $T_p(G\cdot p)=T_gG\times\{0\}$, it is clear that
\begin{gather}\label{Eq:Conn2}
(\Phi\times\text{id})(T_p(G\cdot p)) \subseteq T_p(G\cdot p).
\end{gather}
Therefore:
\begin{align*}
\widetilde\Phi\big(T_{m}(G\cdot m)\big)
&=\widetilde\Phi \big( \varpi(T_p(G\cdot p) \big) \qquad \text{by } (\ref{Eq:Conn1})\\
&=\varpi\circ(\Phi\times\text{id})\big(T_p(G\cdot p) \big) \qquad \text{by } (\ref{Diag:Connection}) \\
&\subseteq \varpi (T_p(G\cdot p)) \qquad \text{by } (\ref{Eq:Conn2}) \\
&= T_{m}(G\cdot m).
\end{align*}
Consequently, since $X(m)\in T_m(G\cdot m)$, the vertical projection of the vector $X(m)$ is such that
\begin{gather*}
\nu(X)(m)= \widetilde\Phi (X(m))\in T_{m}(G\cdot m).
\end{gather*}
Hence, the point $m$ is a relative equilibrium of the vector field $\nu(X)$.
\end{proof}

Now we can prove that the functors of Definition \ref{Def:ExtensionFunctor} and Definition \ref{Def:ProjectionFunctor} also preserve relative equilibria. Parts~(3) and~(4) of the following proposition are especially relevant to the following section's main theorem (Theorem~\ref{Thm:MainTheorem1}).

\begin{Proposition}\label{Prop:RelEqPres}
Let $(M,G,m,V,j)$ be a proper $G$-manifold with slice, let $E\colon \bbx(G_m\times V\toto V)\to \bbx(G\times G\cdot V \toto G\cdot V)$ be the equivariant extension functor $($Definition~{\rm \ref{Def:ExtensionFunctor})}, and let $P\colon \bbx(G\times G\cdot V \toto G\cdot V)\to \bbx(G_m\times V \toto V)$ be the projection functor corresponding to a choice of $G_m$-equivariant splitting $\ffg=\ffg_m\oplus \ffq$.
Then the following are true:
\begin{enumerate}\itemsep=0pt
\item[$1.$] If a point $v\in V$ is a relative equilibrium of an equivariant vector field $X$ on the slice $V$, the point $j(v)$ is a relative equilibrium of the vector field $E(X)$ on the tube $G\cdot V$.
\item[$2.$] If a point $j(v)\in j(V)$ is a relative quilibrium of an equivariant vector field $X$ on the tube $G\cdot V$, the point $v\in V$ is a relative equilibrium of the vector field $P(X)$ on the slice $V$.
\item[$3.$] If the point $m$, through which the slice was chosen, is a relative equilibrium of an equivariant vector field $X$ on the tube $G\cdot V$, then the point $j^{-1}(m)$ is an equilibrium of the vector field~$P(X)$ on the slice~$V$.
\item[$4.$] Let the point $m$ be a $G$-relative equilibrium of a $G$-equivariant vector field $X$ on the tube $G\cdot V$, let $H$ be a Lie subgroup of the stabilizer $G_m$, and let $Y$ be an $H$-equivariant vector field on the slice $V$ that is $H$-isomorphic to the vector field $P(X)$. Then the point $j^{-1}(m)$ is an equilibrium of the vector field $Y$.
\end{enumerate}
\end{Proposition}

\begin{proof} Parts (1) and (2) are a consequence of the fact that pullbacks and pushforwards (when these are defined), equivariant extension of vector fields, and the vertical projection of Defini\-tion~\ref{Def:SplitConnProj} preserve relative equilibria (see Lemmas~\ref{Lemma:PBRelEq},~\ref{Lemma:ExtRelEq}, and~\ref{Lemma:VertRelEq}).

For part (3), let the map
\begin{gather*}
\nu\colon \ \Gamma\big(T\big(G\times^KA\big)\big)^G\to\Gamma\big(\calv\big(G\times^KA\big)\big)^G
\end{gather*}
be the vertical projection of equivariant vector fields induced by the Lie algebra splitting (Definition~\ref{Def:SplitConnProj}). Recall that the tangent space at the point $m$ splits as
\begin{equation*}
T_m(G\cdot V)= T_m(G\cdot m)\oplus T_mj(V),
\end{equation*}
because $V$ is a slice through the point $m$ (see Remark~\ref{Rem:SliceFacts}). Hence, the vector $\nu(X)(m)$ is zero because $X(m)\in T_m(G\cdot m)$ since the point $m$ is a relative equilibrium of the vector field~$X$. Now note that the vector fields $P(X)$ and $\nu(X)$ are $j$-related by definition. Thus
\begin{gather*}
(\d j)P(X)\big(j^{-1}(m)\big)=\nu(X)(m)=0.
\end{gather*}
Since the map $j$ is an embedding, the tangent map $\d j\colon TV\to T(G\cdot V)$ is fiberwise injective. Consequently, $P(X)(j^{-1}(m))=0$, so the point $j^{-1}(m)$ is an equilibrium of the vector field~$P(X)$.

For part (4), let $Y$ be an $H$-equivariant vector field on the slice $V$ and let $\psi\in C^\infty(V,\ffh)^H$ be a map such that $Y=P(X)+\psi_V$.
Recall that there exists a slice $V'$ through the point~$m$, with corresponding $G_m$-equivariant embedding $j'\colon V' \to M$, such that $V'$ is an open ball around the origin $j'^{-1}(m)$ in a vector space with a linear representation of the stabilizer $G_m$ (see Remark~\ref{Rem:SliceExistence}). After perhaps shrinking the slices, there exists a $G_m$-equivariant diffeomorphism $\phi\colon V\to V'$ taking $j^{-1}(m)$ to $j'^{-1}(m)$ (see Remark~\ref{Rem:SliceProjIndep}). Note that the vector fields $\phi_*Y$ and $\phi_*P(X)$ on $V'$ are $H$-isomorphic (the isomorphism is given by the map~$\psi\circ\phi^{-1}$). Furthermore, the point $j'^{-1}(m)$ is an equilibrium of $\phi_*P(X)$ since the point $j^{-1}(m)$ is an equilibrium of $P(X)$ by part~(3). If the vector field $\phi_*Y$ has an equilibrium at $j'^{-1}(m)$, then the vector field $Y$ has an equilibrium at the point~$j^{-1}(m)$. Therefore, it is of no loss of generality to suppose that the slice $V$ is an open ball around the origin $j^{-1}(m)$ in a vector space with a linear representation of~$G_m$. With this assumption, note that
\begin{align*}
Y\big(j^{-1}(m)\big)
&= P(X)\big(j^{-1}(m)\big) + \psi_V\big(j^{-1}(m)\big) \\
&= \psi_V\big(j^{-1}(m)\big) \qquad \text{by part (3)} \\
&= \frac{\d}{\d t}\Big|_0 \exp(t\psi(\mathbf{0}))\cdot \mathbf{0} \qquad \text{since }j^{-1}(m)=\mathbf{0}\\
&= \frac{\d}{\d t}\Big|_0 \mathbf{0} \qquad \text{since the action is linear} \\
&= 0.
\end{align*}
Hence, the point $j^{-1}(m)$ is an equilibrium of the vector field $Y$.
\end{proof}

\begin{Remark}\label{Rem:ProjectionFunctorSubgroup}
Let $(M,G,m,V,j)$ be a proper $G$-manifold with slice.
In this paper, we will sometimes consider vector fields that are equivariant only with respect to a Lie subgroup $H$ of the full symmetry group $G$.
We view these as objects of the groupoid
\begin{gather*}
\bbx(H\times G\cdot V\toto G\cdot V):=C^\infty(G\cdot V,\ffh)^H\ltimes\Gamma(T(G\cdot V))^H
\end{gather*}
of $H$-equivariant vector fields in the tube $G\cdot V$.
Given a $G_m$-equivariant splitting $\ffg=\ffg_m\oplus\ffq$, the corresponding $G$-equivariant connection $\widetilde\Phi\in\Omega^1(G\cdot V ; \calv(G\cdot V))$ is also $H$-equivariant. As stated in Remark~\ref{Rem:VertProjIsEquiv}, the vertical projection of Definition~\ref{Def:SplitConnProj} takes $H$-equivariant vector fields to $H$-equivariant vector fields. However, we need to generalize the projection functor of Definition~\ref{Def:ProjectionFunctor} to handle the morphisms of the groupoid $\bbx(H\times G\cdot V \toto G\cdot V)$. For this, we make a~choice of $H_m$-equivariant splitting $\ffh=\ffh_m\oplus\ffp$. This gives an $H_m$-equivariant projection map $\bbp_H\colon \ffh\to\ffh_m$. With this we can generalize the projection functor of Definition~\ref{Def:ProjectionFunctor}.
\end{Remark}

\begin{Definition}[projection functor with respect to a subgroup]\label{Def:ProjectionFunctorSubgroup}
Let $(M,G,m,V,j)$ be a proper $G$-manifold with slice, let $H$ be a Lie subgroup of $G$, and suppose you are given splittings $\ffg=\ffg_m\oplus\ffq$ and $\ffh=\ffh_m\oplus\ffp$ that are $G_m$-equivariant and $H_m$-equivariant respectively. Let $\nu\colon \Gamma(T(G\cdot V))^H\to\Gamma(\calv(G\cdot V))^H$ be the vertical projection of $H$-equivariant vector fields induced by the splitting of $\ffg$ and let $\bbp_H\colon \ffh\to\ffh_m$ be the projection induced by the splitting of~$\ffh$. The {\it projection functor with respect to the subgroup $H$} is the functor
\begin{gather*}
\xy
(-25,10)*+{P_H\colon \bbx(H\times G\cdot V\toto G\cdot V)}="1";
(25,10)*+{\bbx(H_m\times V\toto V),}="2";
{\ar@{->} "1";"2"};
(-45,0)*+{\big(X}="3";
(-25,0)*+{Y\big)}="4";
{\ar@{->}^{(\psi,X),} "3";"4"};
(0,0)*+{(P_H(X)}="5";
(38,0)*+{P_H(Y)),}="6";
{\ar@{->}^{(P_H(\psi),P_H(X))} "5";"6"};
{\ar@{|->} (-15,0);(-10,0)};
\endxy
\end{gather*}
where for any vector field $X\in\Gamma(T(G\cdot V))^G$ we define
\begin{gather*}
P_H(X)\colon \ V\to TV, \qquad P_H(X):= j^*(\nu(X)),
\end{gather*}
and for any map $\psi\in C^\infty(G\cdot V,\ffh)^H$ we define
\begin{gather*}
P_H(\psi)\colon \ V\to\ffh_m, \qquad P_H(\psi):=j^*(\bbp_H\circ\psi).
\end{gather*}
\end{Definition}

\section{Stability of relative equilibria}\label{Stability}

In this section we show that stability is preserved by isomorphisms of equivariant vector fields (Proposition~\ref{Prop:IVFsSameGStable}), opening the door to replacing the given vector field with an isomorphic one that is potentially easier to work with. The main result of this section (Theorem~\ref{Thm:MainTheorem1}) is a~stability test that involves passing from the category of equivariant vector fields on a tube to the category of equivariant vector fields on a slice, which is also easier to work with.

We begin by recalling the following definition of nonlinear stability in a proper $G$-manifold due to Patrick~\cite{P91,P95}:

\begin{Definition}[stability modulo a subgroup]\label{Def:Stability}
Let $M$ be a $G$-manifold, let $X$ be a $G$-equivariant vector field on $M$, and let $H\le G$ be a Lie subgroup of~$G$. A $G$-relative equilibrium $m\in M$ of the vector field $X$ is {\it $H$-stable}, or {\it stable modulo}~$H$, if for any $H$-invariant neighborhood $U\subseteq M$ of the point $m$ there exists a neighborhood $O\subseteq U$ of the point $m$ for which all maximal integral curves of the vector field $X$ starting at points in the neighborhood~$O$ stay in the neighborhood~$U$ for all times for which they are defined.
\end{Definition}

\begin{Remark}\label{Rem:StabForFlows}
Let $X$ be a smooth vector field on a $G$-manifold $M$, and let $\phi\colon A\to M$ be its flow (Definition~\ref{Def:Flow}). Stability modulo a~subgroup $H$ (Definition~\ref{Def:Stability}) can be rephrased as saying that the relative equilibrium $m$ of the vector field $X$ is $H$-stable if for all $H$-invariant neighborhoods $U$ of the point $m$, there exists a neighborhood $O\subseteq U$, containing the point $m$, for which the flow $\phi$ of the vector field $X$ satisfies $\phi(t,q)\in U$ for all pairs $(t,q)\in A$ with $q\in O$.
\end{Remark}

The following fact about $G$-stability will be useful later:

\begin{Lemma}\label{Lemma:SubgroupStability}
Let $M$ be a $G$-manifold, let $X$ be a $G$-equivariant vector field on the mani\-fold~$M$, let $m\in M$ be a $G$-relative equilibrium of $X$, and let $H\le K$ be Lie subgroups of~$G$. If the $G$-relative equilibrium $m$ is $H$-stable, then it is $K$-stable.
\end{Lemma}

\begin{proof}
Any $K$-invariant neighborhood $U$ of the point $m$ is in particular $H$-invariant since $H\le K$. Hence, we can find the required neighborhood $O\subseteq U$ by using the $H$-stability of the point~$m$.
\end{proof}

We now show that the stability of relative equilibria is preserved by morphisms of equivariant vector fields:

\begin{Proposition}\label{Prop:IVFsSameGStable}
Let $M$ be a proper $G$-manifold and let $X$ and $Y$ be two isomorphic equivariant vector fields on~$M$. If a point $m\in M$ is a $G$-stable relative equilibrium of the vector field $X$, then it is a $G$-stable relative equilibrium of the vector field~$Y$.
\end{Proposition}

\begin{proof}
Let $\phi^X$ and $\phi^Y$ be the maximal flows of the vector fields $X$ and $Y$ respectively. By Theorem \ref{Thm:IsoFlows} we know there exists a family of smooth maps $\{F_t\colon M\to G\}$, depending smoothly on $t$, such that
\begin{gather*}
\phi^Y(t,q)=F_t(q)\cdot \phi^X(t,q)
\end{gather*}
for all pairs $(t,q)$ for which $\phi^X$ is defined. Recall that this also shows the flows $\phi^X$ and $\phi^Y$ have the same domain (see Corollary~\ref{Cor:SameOrbitFlow}).

Now let $U\subseteq M$ be a $G$-invariant open neighborhood of the relative equilibrium~$m$. We seek a neighborhood $O\subseteq U$ of the point $m$ such that all maximal integral curves of $Y$ starting at points in $O$ stay in $U$ for all times in their domain. Since the point $m$ is $G$-stable for the vector field $X$, we know there exists a neighborhood $O\subseteq U$ of the point $m$ for which all integral curves of $X$ starting at points of $O$ stay in $U$ for all time. This means that for any point $q\in O$ and all times $t$ for which $(t,q)$ is in the domain of $\phi^Y$, we have that
\begin{gather*}
\phi^Y(t,q)=F_t(q)\cdot \phi^X(t,q)\in F_t(q)\cdot U=U,
\end{gather*}
where the last equality holds since the neighborhood $U$ is $G$-invariant. Thus, the relative equilibrium $m$ is $G$-stable for the vector field~$Y$.
\end{proof}

There is also a notion of stability for fixed points of abstract flows:

\begin{Definition}\label{Def:AbsFlowStability}
Let $z$ be a fixed point of an abstract flow $\phi\colon A\to Z$ on a topological space~$Z$. The point $z$ is {\it stable} if for all neighborhoods $U\subseteq Z$ of the point~$z$, there exists a neighborhood $O\subseteq U$, containing the point $z$, such that for all pairs $(t,q)\in A$ with $q\in O$, we have that $\phi(t,q)\in U$.
\end{Definition}

Next, we relate $G$-stability on a proper $G$-manifold with stability on the orbit space:

\begin{Lemma}\label{Lemma:OrbitSpaceStability}
Let $M$ be a proper $G$-manifold, let $\pi\colon M\to M/G$ be the quotient map, and let $X$ be an equivariant vector field on $M$ with the point $m$ as a relative equilibrium. Let $\phi\colon A\to M$ be the maximal flow of $X$ and $\Phi\colon B\to M/G$ the induced abstract flow on the orbit space $($Lemma~{\rm \ref{Lemma:InducedOrbitFlow})}. Then the relative equilibrium $m\in M$ is $G$-stable for the vector field~$X$ $($in the sense of Definition~{\rm \ref{Def:Stability})} if and only if the fixed point $\pi(m)\in M/G$ is stable for the induced flow~$\Phi$ on the orbit space $($in the sense of Definition~{\rm \ref{Def:AbsFlowStability})}.
\end{Lemma}

\begin{proof} First, suppose that the point $m$ is a $G$-stable relative equilibrium of the vector field $X$ and let $U\subseteq M/G$ be an open neighborhood of the point~$\pi(m)$. We seek an open neighborhood $O\subseteq U$ of $\pi(m)$ such that the curves of the abstract flow starting at points in~$O$ stay in~$U$ for all times in their domain. Note that $\pi^{-1}(U)\subseteq M$ is a $G$-invariant open neighborhood of the point~$m$. Since the relative equilibrium $m$ of the vector field $X$ is $G$-stable, there exists a neighborhood $V\subseteq \pi^{-1}(U)$ such that all maximal integral curves of $X$ starting at points of~$V$ stay in $\pi^{-1}(U)$ for all times in their domain. We need a $G$-invariant, and hence saturated, neighborhood of the point~$m$ with the same properties as~$V$. Define
\begin{gather*}
W:=\bigcup_{g\in G}g_M(V).
\end{gather*}
As desired, this set is open, it is contained in $\pi^{-1}(U)$, and all maximal integral curves of $X$ starting at points in $W$ stay in $\pi^{-1}(U)$ for all times in their domain. This set is open since it is the union of the sets $g_M(V)$, each of which are in turn open because the group translations $g_M\colon M\to M$ are diffeomorphisms. We know that the neighborhood $W$ is contained in the neighborhood $\pi^{-1}(U)$ because $V\subseteq \pi^{-1}(U)$ and the neighborhood $\pi^{-1}(U)$ is $G$-invariant. To verify the last property, let $q\in V$ and $g\in G$, so that $g\cdot q\in W$ is an arbitrary point in $W$. By the choice of $V$, the maximal integral curve $\phi(\cdot,q)$ of~$X$ starting at $q$ stays in~$\pi^{-1}(U)$ for all times for which it is defined. Thus, by the $G$-equivariance of the flow and the $G$-invariance of~$\pi^{-1}(U)$
\begin{gather*}
\phi(t,g\cdot q)=g\cdot \phi(t,q)\in g_M\big(\pi^{-1}(U)\big)=\pi^{-1}(U)
\end{gather*}
for all times $t$ such that $(t,g\cdot q)\in A$. Hence, $W$ is as claimed.

Now consider the set $O:=\pi(W)\subseteq \pi(\pi^{-1}(U))\subseteq U$. The set $O$ is an open neighborhood of~$\pi(m)$ since the set $W$ is a~$G$-invariant, and hence saturated, open neighborhood of $m$ in $M$. It is contained in $U$ since $\pi(W)\subseteq\pi(\pi^{-1}(U))\subseteq U$. Furthermore, the curves of the abstract flow starting at points in $O$ stay in $U$ for all times in their domain. To verify this last statement, let $q\in W$, so that $\pi(q)\in O$ is an arbitrary point in~$O$. Observe that, by diagram (\ref{Diag:OrbitFlow}), we have that
\begin{gather*}
\Phi(t,\pi(q))=\pi(\phi(t,q))
\end{gather*}
for all $t$ such that $(t,\pi(q))\in B$. Hence, by the choice of $W$, for all times $t$ such that $(t,\pi(q))\in B$, we have that
\begin{gather*}
\Phi(t,\pi(q))=\pi(\phi(t,q))\in\pi\big(\pi^{-1}(U)\big)\subseteq U.
\end{gather*}
Hence, the point $\pi(m)$ is stable in the sense of Definition~\ref{Def:AbsFlowStability}.

Conversely, let the point $\pi(m)$ be stable in the sense of Definition~\ref{Def:AbsFlowStability} and let $U\subseteq M$ be an open $G$-invariant neighborhood of the relative equilibrium~$m$. We seek an open neighborhood $O\subseteq U$ of the point $m$ such that the maximal integral curves of $X$ starting at points in $O$ stay in $U$ for all times in their domain. Since the open set $U$ is $G$-invariant, and hence saturated, the set $\pi(U)\subseteq M/G$ is an open neighborhood of the point $\pi(m)$. The stability of the point $\pi(m)$ implies that there exists a neighborhood $V\subseteq \pi(U)$ of the point $\pi(m)$ such that the curves of the abstract flow starting at points of in $V$ stay in $\pi(U)$ for all times in their domain.

Now consider the open neighborhood $O:=\pi^{-1}(V)$ of the point $m$. Note that the neighborhood $O$ is contained in $U$ since $O=\pi^{-1}(V)\subseteq\pi^{-1}(\pi(U))=U$; where we use that $U$ is $G$-invariant. Furthermore, for all points $q\in O$ the maximal integral curve $\phi(\cdot,q)$ is such that
\begin{gather*}
\pi(\phi(t,q))=\Phi(t,\pi(q))\in \pi(U)
\end{gather*}
for all times $t$ with $(t,q)\in A$. Thus, by the $G$-invariance of the neighborhood $U$, we know that
\begin{gather*}
\phi(t,q)\in \pi^{-1}(\pi(U))=U
\end{gather*}
for all times $t$ with $(t,q)\in A$. Hence, the relative equilibrium $m$ of the vector field $X$ is $G$-stable in the sense of Definition~\ref{Def:Stability}.
\end{proof}

\begin{Remark} Lemma \ref{Lemma:OrbitSpaceStability} says that stability of a relative equilibrium reduces to stability of the fixed point of the induced flow on the orbit space. If the orbit space is a manifold, for example when the action is free and proper, then one can appeal to the vast literature on stability of fixed points to test for stability. However, if the action is not free, the orbit space is in general not a manifold.
In that case we must appeal to other arguments like the ones presented in this paper. Proposition~\ref{Prop:IVFsSameGStable} is key in doing this.
\end{Remark}

Lemma \ref{Lemma:OrbitSpaceStability} provides another way to show Proposition \ref{Prop:IVFsSameGStable}:

\begin{proof}
By Lemma~\ref{Lemma:SameRelEqs}, the relative equilibrium $m$ of the vector field $X$ is also a relative equilibrium of the vector field $Y$.
By Lemma~\ref{Lemma:OrbitSpaceStability}, the $G$-stability of the relative equilibrium~$m$ of $Y$ corresponds to the stability (in the sense of Definition~\ref{Def:AbsFlowStability}) of the corresponding fixed point of the induced orbit space flow. On the other hand, by Corollary~\ref{Cor:SameOrbitFlow}, the vector fields $X$ and $Y$ induce the same abstract flow on the orbit space. Thus, the $G$-stability of the relative equilibrium $m$ for the vector field $X$ implies the stability of the corresponding fixed point of the abstract flow on the orbit space induced by~$Y$. Hence, the relative equilibrium $m$ of the vector field $Y$ is $G$-stable.
\end{proof}

To prove the slice stability criterion (Theorem~\ref{Thm:MainTheorem1}) we need to show that the equivariant extension functor of Definition~\ref{Def:ExtensionFunctor} preserves the stability of relative equilbiria. For this, we need the following two lemmas:

\begin{Lemma}\label{Lemma:fRelatedStability}
Let $M$ and $N$ be proper $G$-manifolds and let $f\colon M\to N$ be a $G$-equivariant diffeomorphism. Suppose that $X$ and $Y$ are $f$-related equivariant vector fields on $M$ and $N$ respectively. Then a point $m$ is a $G$-stable $G$-relative equilibrium of the vector field $X$ if and only if the point $f(m)$ is a $G$-stable $G$-relative equilibrium of the vector field $Y$. In particular, the pushforward and pullback of vector fields by the diffeomorphism $f$ preserve stability of relative equilibria.
\end{Lemma}

\begin{proof}
Suppose first that the relative equilibrium $m$ is $G$-stable. Let $U\subseteq N$ be a $G$-invariant neighborhood of the point $f(m)$. We seek a neighborhood $O\subseteq U$ of the point $f(m)$ such that all maximal integral curves of $Y$ starting at points in $O$ stay in $U$ for all times in their domain. By the equivariance of $f$ and the $G$-invariance of the set $U$, the open set $f^{-1}(U)$ is a $G$-invariant neighborhood of the point~$m$. By the $G$-stability of the point $m$, there exists a neighborhood $W\subseteq f^{-1}(U)$ of the point $m$ such that the maximal integral curves of~$X$ starting at points of~$W$ stay in the set $U$ for all times in their domain.

Consider the set $O:=f(W)$. It is open since the map $f$ is a diffeomorphism. It is contained in the neighborhood $U$ since $f$ is a diffeomorphism and $W\subseteq f^{-1}(U)$. Consider an arbitrary point $q\in O$ and let $\gamma_q$ be the maximal integral curve of the vector field $Y$ starting at the point~$q$. Since the vector fields $X$ and $Y$ are $f$-related, the curve~$f^{-1}\circ \gamma_q$ is the maximal integral curve of~$X$ starting at the point $f^{-1}(q)\in W$, and it is defined for the same times that $\gamma_q$ is. By the choice of~$W$, we know that $f^{-1}(\gamma_q(t))\in f^{-1}(U)$ for all times $t$ such that the curve is defined. Hence, $\gamma_q(t)=f(f^{-1}(\gamma_q(t)))\in f(f^{-1}( U))=U$ for all times $t$ for which the curve is defined. Therefore, the relative equilibrium $f(m)$ is $G$-stable for~$Y$. The converse is completely analogous.
\end{proof}

\begin{Lemma}\label{Lemma:ExtensionStability}
Let $M$ be a proper $G$-manifold, let $K$ be a Lie subgroup of $G$, and let $A$ be a $K$-invariant regular submanifold of $M$ satisfying the hypotheses of Lemma~{\rm \ref{Lemma:EquivExtension}}. Suppose that $X$ is a $K$-equivariant vector field on $A$ and that the point $a\in A$ is a $K$-stable $K$-relative equilibrium of~$X$. Then the point $a$ is a $G$-stable $G$-relative equilibrium of the equivariant extension $\varepsilon X$ of~$X$.
\end{Lemma}

\begin{proof} By Lemma \ref{Lemma:ExtRelEq}, we know that the point $a$ is a relative equilibrium of the equivariant extension $\varepsilon X$. Hence, it remains to show that the relative equilibrium is $G$-stable. Let $\iota\colon A\hookrightarrow G\cdot A$ be the inclusion map of the submanifold $A$, and let $U\subseteq G\cdot A$ be an arbitrary $G$-invariant neighborhood of the point $a$. We seek an open neighborhood $O\subseteq U$ of the point $a$ such that all maximal integral curves of the vector field $\varepsilon X$ starting at points in $O$ stay in $U$ for all time. Observe that the set $U\cap A$ is a $K$-invariant neighborhood in the subspace topology of $A$. Hence, there exists an open set $W$ in $G\cdot A$ such that $W\cap A$ is contained in $U\cap A$, and the maximal integral curves of $X$ starting at points in the set $W\cap A$ stay in $U\cap A$ for all times in their domains.

Consider the saturation $O:=G\cdot W=G\cdot (W\cap A)$. This set is open in $G\cdot A$ since it is the union, over all elements $g\in G$, of the open sets $g_M(W)$. Also note that for all $g\in G$ we have $g_M(W\cap A)\subseteq g_M(U\cap A)=U$, where the last equality uses the $G$-invariance of $U$. Hence, the neighborhood $O$ is contained in $U$.

Now let $w\in W\cap A$ and $g\in G$, so that $q=g\cdot w\in O$ is an arbitrary point in $O$. Let~$\gamma_q$ and~$\gamma_w$ be the maximal integral curves of the equivariant extension $\varepsilon X$ starting at the points~$q$ and~$w$ respectively. Since the vector fields $\varepsilon X$ and $X$ are $\iota$-related, the curve $\gamma_w$ is also the maximal integral curve of the vector field $X$ starting at the point~$w$. Consequently, by the choice of $W$ and the $K$-stability of the relative equilibrium $a$ of $X$, we have that $\gamma_w(t)\in U\cap A$ for all times~$t$ for which it is defined. This and the $G$-equivariance of the flow of $\varepsilon X$ implies that, for all times~$t$ for which the integral curve~$\gamma_q$ is defined, we have that
\begin{gather*}
\gamma_q(t)=g\cdot \gamma_w(t)\in g_M(U\cap A)=U,
\end{gather*}
where we also use the $G$-invariance of the set $U$. Therefore, the point $a$ is $G$-stable for the vector field $\varepsilon X$.
\end{proof}

\begin{Proposition}\label{Prop:EStability}
Let $(M,G,m,V,j)$ be a proper $G$-manifold with slice, let $E\colon \bbx(G_m\times V\toto V)\to \bbx(G\times G\cdot V \toto G\cdot V)$ be the equivariant extension functor $($Definition~{\rm \ref{Def:ExtensionFunctor})}, and let $X$ be a $G_m$-equivariant vector field on the slice $V$.
Then if a point $v\in V$ is a $G_m$-stable $G_m$-relative equilibrium of the vector field $X$, the point $j(v)$ is a $G$-stable $G$-relative equilibrium of the vector field $E(X)$ on the tube $G\cdot V$. That is, the functor $E$ preserves stability of relative equilibria.
\end{Proposition}

\begin{proof} We have already shown that the point $j(v)$ is a $G$-relative equilibrium of the vector field~$E(X)$ in part~(1) of Proposition~\ref{Prop:RelEqPres}, so it remains to show the statement concerning stability. Apply Lemma~\ref{Lemma:fRelatedStability} to the $G_m$-equivariant diffeomorphism $j\colon V\to j(V)$. This shows that the point $j(v)$ is $G_m$-stable for the vector field $j_*X$ on the regular submanifold $j(V)$. Now apply Lemma~\ref{Lemma:ExtensionStability} with $K=G_m$, $A=j(V)$, and the vector field $j_*X$. This shows that the point~$j(v)$ is $G$-stable for the vector field $E(X)=\varepsilon j_*X$ on the tube $G\cdot V$, which is what we wanted to prove.
\end{proof}

The following theorem provides a criterion for $G$-stability of a relative equilibrium and is the main theorem of this section:

\begin{Theorem}[slice stability criterion]\label{Thm:MainTheorem1}
Let $(M,G,m,V,j)$ be a proper $G$-manifold with slice, let $P\colon \bbx(G\times G\cdot V\toto G\cdot V)\to \bbx(G_m\times V\toto V)$ be the projection functor corresponding to a~choice of $G_m$-equivariant splitting $\ffg=\ffg_m\oplus\ffq$ $($Definition~{\rm \ref{Def:ProjectionFunctor})}, and let $X$ be a~$G$-equivariant vector field on the tube $G\cdot V$ with the point $m$ as a~$G$-relative equilibrium. Suppose there exists a~Lie subgroup $H$ of the stabilizer $G_m$ and an $H$-equivariant vector field $Y$ on the slice $V$ such that
\begin{enumerate}\itemsep=0pt
\item[$1)$] the vector field $Y$ is $H$-isomorphic to the projected vector field $P(X)$;
\item[$2)$] the point $j^{-1}(m)\in V$ is an $H$-stable equilibrium of the vector field $Y$.
\end{enumerate}
Then the $G$-relative equilibrium $m$ of the vector field $X$ is $G$-stable.
\end{Theorem}

\begin{Remark}\label{Rem:OnMainThm1} Before proceeding with the proof of Theorem \ref{Thm:MainTheorem1}, we make some observations:
\begin{enumerate}\itemsep=0pt
\item By Proposition \ref{Prop:RelEqPres}, we know that the point $j^{-1}(m)$ is an equilibrium of the vector field~$P(X)$. The same proposition also shows that the point $j^{-1}(m)$ is an equilibrium of any vector field isomorphic to $P(X)$ with respect to any Lie subgroup of the stabilizer~$G_m$.
\item Since we can choose the slice $V$ to be an open ball around the origin in a finite dimensional vector space with a representation of the stabilizer $G_m$ (see part (2) of Remark~\ref{Rem:SliceFacts}), Theo\-rem~\ref{Thm:MainTheorem1} reduces determining nonlinear stability of a~relative equilibrium on a manifold to determining stability of an equilibrium on a vector space. Furthermore, it replaces dealing with an action of a possibly noncompact Lie group with dealing with a representation of a compact Lie group. This is desirable since there is a vast literature available on the stability of equilibria on vector spaces with representations of compact Lie groups (see, for example, Field~\cite{F07} and the references therein).
\item Any choice of slice or projection leads to isomorphic projected vector fields (see Remark~\ref{Rem:SliceProjIndep}). Hence, one may choose any convenient slice and projection to apply Theo\-rem~\ref{Thm:MainTheorem1}.
\item Given a $G$-equivariant vector field with a $G$-relative equilibrium for which we want to test for stability, Theorem~\ref{Thm:MainTheorem1} allows us to break the symmetry of the projected vector field by considering an isomorphic vector field with respect to a subgroup of the symmetry group. Such a replacement of the given vector field by an isomorphic one will be illustrated in our application (Theorem~\ref{Thm:MainTheorem2}) in Section~\ref{HamStability}.
\end{enumerate}
\end{Remark}

We now proceed with the proof of Theorem \ref{Thm:MainTheorem1}:

\begin{proof} Let $E\colon \bbx(G_m\times V\toto V)\to \bbx(G\times G\cdot V\toto G\cdot V)$ be the equivariant extension functor (see Definition~\ref{Def:ExtensionFunctor}). The proof uses that stability with respect to a subgroup implies stability with respect to the larger group (Lemma~\ref{Lemma:SubgroupStability}), that morphisms of equivariant vector fields and the equivariant extension functor preserve stability (Propositions~\ref{Prop:IVFsSameGStable} and~\ref{Prop:EStability}), and that the functors $E$ and $P$ form an equivalence of categories (Theorem~\ref{Thm:LermanEquivalence}). The argument is as follows.

Since the vector fields $Y$ and $P(X)$ are $H$-isomorphic, Proposition \ref{Prop:IVFsSameGStable} implies that the point~$j^{-1}(m)$ is $H$-stable for the vector field~$P(X)$. Since~$H$ is a subgroup of the stabilizer~$G_m$, Lemma~\ref{Lemma:SubgroupStability} implies that the point~$j^{-1}(m)$ is $G_m$-stable for the vector field $P(X)$. By Proposition~\ref{Prop:EStability}, the point $m$ is $G$-stable for the equivariant extension $E(P(X))$. Now note that the vector fields~$X$ and~$E(P(X))$ are $G$-isomorphic by Theorem~\ref{Thm:LermanEquivalence}.
Thus, by Lemma \ref{Lemma:SameRelEqs}, the $G$-relative equilibrium $m$ of the vector field $X$ is also a $G$-relative equilibrium of the vector field~$E(P(X))$. Finally, by Proposition \ref{Prop:IVFsSameGStable}, the $G$-relative equilibrium $m$ of the vector field $X$ is $G$-stable since it is $G$-stable for the isomorphic vector field $E(P(X))$.
\end{proof}

\section{Stability of Hamiltonian relative equilibria}\label{HamStability}

In this section we apply the slice stability criterion (Theorem \ref{Thm:MainTheorem1}) to Hamiltonian relative equilibria. We recall the definition of an equivariant momentum map:

\begin{Definition}\label{Def:MomentumMap}
Let $M$ be a symplectic manifold with an action of a Lie group $G$ by Hamiltonian symplectomorphisms. Suppose the symplectic form $\omega$ is $G$-invariant. A smooth map $\Phi\colon M\to \ffg^*$ is an equivariant {\it momentum map} for the action of $G$ on the symplectic manifold $M$ if:
\begin{itemize}\itemsep=0pt
\item the map $\Phi$ is equivariant with respect to the given action on $M$ and the coadjoint representation on $\ffg^*$;
\item for all vectors $\xi\in\ffg$, the function
\begin{gather*}
\langle\Phi,\xi\rangle\colon \ M\to\bbr, \qquad m\mapsto \langle\Phi(m),\xi\rangle,
\end{gather*}
where $\langle\cdot,\cdot\rangle\colon \ffg^*\times\ffg\to\bbr$ is the pairing of the Lie algebra $\ffg$ and its dual $\ffg^*$, is a Hamiltonian function for the fundamental vector field $\xi_M$.
\end{itemize}
\end{Definition}

Throughout this section we will work in the following settings:

\begin{Definition}\label{Def:HamGSpace}
A {\it Hamiltonian $G$-space} is a quadruple $(M,\omega,G,\Phi)$ where $M$ is a proper $G$-manifold such that
\begin{itemize}\itemsep=0pt
\item[1)] the manifold $M$ is symplectic with corresponding $G$-invariant symplectic form $\omega$;
\item[2)] the action of the Lie group $G$ is by Hamiltonian symplectomorphisms;
\item[3)] the map $\Phi\colon M\to\ffg^*$ is an equivariant momentum map for the action.
\end{itemize}
\end{Definition}

\begin{Definition}\label{Def:HamGSystem} A {\it Hamiltonian $G$-system} is a quintuple $(M,\omega, G, \Phi, h)$ where the quadruple $(M,\omega,G,\Phi)$ is a Hamiltonian $G$-space and the function $h\colon M\to\bbr$ is a smooth $G$-invariant function called the {\it Hamiltonian function} of the system.
\end{Definition}

We are interested in the stability of $G$-relative equilibria of Hamiltonian vector fields. In particular, we determine stability with respect to the following group:

\begin{Definition}\label{Def:MomentIsotropy} Let $(M,\omega,G,\Phi)$ be a Hamiltonian $G$-space and let $m$ be a point in the manifold~$M$. The covector $\mu:=\Phi(m)\in\ffg^*$ is called the {\it moment} of the point~$m$, and the Lie subgroup of $G$ defined by
\begin{gather*}
G_\mu:=\big\{g\in G\,|\, \CoAd(g)\mu=\mu\big\}
\end{gather*}
is called the {\it moment isotropy group} of the point $m$. We denote by $\ffg_\mu$ the Lie algebra of the moment isotropy group.
\end{Definition}

Hamiltonian relative equilibria of Hamiltonian $G$-systems have the following well-known cha\-rac\-terization (see, for example, \cite[Theorem~4.1]{M92}):

\begin{Lemma}\label{Lemma:HamRelEqChar}
Let $(M,\omega,G,\Phi,h)$ be a Hamiltonian $G$-system, let $\Xi_h\in\Gamma(TM)^G$ be the Hamiltonian vector field of the function $h$, and let $m\in M$ be a point with moment $\mu=\Phi(m)$. The following are equivalent:
\begin{enumerate}\itemsep=0pt
\item[$1)$] the point $m$ is a $G$-relative equilibrium of the Hamiltonian vector field $\Xi_h$;
\item[$2)$] there exists a velocity vector $\xi\in\ffg_\mu$ of the point $m$; that is, $\xi\in\ffg_\mu$ is such that $\Xi_h(m)=\xi_M(m)$;
\item[$3)$] the point $m\in M$ is a critical point of the function $h^\xi:=h-\langle\Phi,\xi\rangle$.
\end{enumerate}
\end{Lemma}

\begin{Definition}\label{Def:AugHam} Let $(M,\omega,G,\Phi,h)$ be a Hamiltonian $G$-system. Given a vector $\xi$ in the Lie algebra $\ffg$, the {\it Hamiltonian function augmented by $\xi$}, or simply the {\it augmented Hamiltonian}, is the function
\begin{gather*}
h^\xi\colon \ M\to\bbr,\qquad h^\xi:=h-\langle\Phi,\xi\rangle.
\end{gather*}
The corresponding Hamiltonian vector field $\Xi_{h^\xi}=\Xi_h-\xi_M$ is called the {\it Hamiltonian vector field augmented by $\xi$}, or simply the {\it augmented Hamiltonian vector field}.
\end{Definition}

\begin{Remark}\label{Rem:Hessians} Given a smooth function $f\in C^\infty(M)$ on a manifold $M$, the Hessian of $f$ is only well-defined at critical points of~$f$. If $m\in M$ is a critical point of $f$, the Hessian $\d^2f(m)\colon T_mM\times T_mM\to \bbr$ of $f$ at the point $m$ behaves well under change of coordinates and pull-backs. That is, if $j\colon N\to M$ is a smooth map with $j(n)=m$, then
\begin{gather*}
\d^2(j^* f)(n)=j^*\big(\d^2f(m)\big).
\end{gather*}
In particular, if $N$ is a submanifold of $M$, then
\begin{gather*}
\big(\d^2f\big)(m)|_{T_mN}=\d^2 (f|_N t)(m).
\end{gather*}
Lemma \ref{Lemma:HamRelEqChar} guarantees that if $m$ is a $G$-relative equilibrium of the Hamiltonian vector field $\Xi_h$, then the augmented hamiltonian $h^\xi$ has a well-defined Hessian at~$m$.
\end{Remark}

Lerman and Singer \cite{LS98} and Ortega and Ratiu \cite{OR99}, building on work of Patrick \cite{P91, P95}, proved a criterion for~$G_\mu$-stability of a Hamiltonian $G$-relative equilibrium involving the Hessian of the augmented Hamiltonian~$h^\xi$. Their work required an orthogonality condition on the velocity~$\xi$. Montaldi and Rodr\'iguez-Olmos were able to eliminate this condition in a generalized criterion; first, for the case of compact moment isotropy in \cite[Theorem~2]{MRO11}; and then, more generally, for the case of possibly noncompact moment isotropy in \cite[Theorem~3.6]{MRO15}. They also provide an example where this criterion predicts stability, while previous criteria were inconclusive~\cite[Remark~3.7]{MRO15}. Their result is as follows:

\begin{Theorem}[{Montaldi and Rodr\'iguez-Olmos \cite[Theorem~3.6]{MRO15}}]\label{Thm:MROTheorem}
Let $(M,\omega,G,\Phi,h)$ be a~Ha\-miltonian $G$-system. Suppose $m\in M$ is a $G$-relative equilibrium of the Hamiltonian vector field~$\Xi_h$ of the function $h$, and let $\mu$ be the moment of the point~$m$. Suppose further that
\begin{enumerate}\itemsep=0pt
\item[$1)$] the moment isotropy subgroup $G_\mu$ acts properly on the manifold $M$;
\item[$2)$] there exists an $\Ad(G_\mu)$-invariant inner product on the Lie algebra $\ffg$ of $G$;
\item[$3)$] there exists a velocity vector $\xi\in\ffg_\mu$ of the point $m$ such that the Hessian $\d^2h^\xi(m)$ is definite and nondegenerate on a~$G_m$-invariant complement $W$ to the tangent space $T_m(G_\mu\cdot m)$ in $\ker\d\Phi_m$.
\end{enumerate}
Then the $G$-relative equilibrium $m$ is $G_\mu$-stable.
\end{Theorem}

\begin{Remark}\label{Rem:Strategy} The stability criteria in \cite{LS98,MRO15, MRO11, OR99} can be seen from the point of view of the isomorphic relationship between the Hamiltonian vector field and its augmented counterparts (the precise isomorphic relationship follows from Lemma~\ref{Lemma:AugIso}). The main goal of this section is to provide an alternative proof of Theorem~\ref{Thm:MROTheorem} from this point of view.
Recall that the slice stability criterion (Theorem~\ref{Thm:MainTheorem1}) uses isomorphisms of equivariant vector fields to obtain a criterion for stability of relative equilibria. In Lemma~\ref{Lemma:Constants}, and Corollaries~\ref{Cor:ConstantsAugHam} and~\ref{Cor:ConstantsVerticalAugHam}, which we will prove next, we describe a relationship between the invariant constants of motion of isomorphic vector fields. We use this relationship in order to apply Theorem~\ref{Thm:MainTheorem1} in our proof of Theorem~\ref{Thm:MROTheorem}.
\end{Remark}

\begin{Remark}\label{Rem:FixedMomentum} To prove Theorem \ref{Thm:MROTheorem}, one can reduce to the case when the momentum of the relative equilibrium is fixed by the coadjoint representation. Equivalently one can reduce to the case when the momentum is zero. This can be achieved by using symplectic cross-sections exactly like in \cite[Section~2.3]{LS98}. This assumption simplifies several arguments in the proof of Theorem~\ref{Thm:MROTheorem}, so we will suppose it holds throughout the rest of the paper. One may also assume without loss of generality that the hamiltonian function is such that $h(m)=0$.
\end{Remark}

\begin{Remark}\label{Rem:NormsOnLie} We will need several norms on finite-dimensional Lie algebras and their duals. Let $G$ be a Lie group, let $\ffg$ be its Lie algebra, and let $\ffg^*$ be the dual of $\ffg$. Recall that a $G$-invariant inner product $(\cdot,\cdot)_\ffg$ on the Lie algebra $\ffg$ induces a $G$-invariant norm $||\cdot||_{\ffg}$ on $\ffg$ given by
\begin{gather*}
||\nu||_{\ffg}:=\sqrt{(\nu,\nu)_\ffg}
\end{gather*}
for all $\nu\in\ffg$.
Let $(\cdot,\cdot)_{\ffg^*}$ be the inner product on $\ffg^*$ dual to the inner product on $\ffg$. Then we also get a $G$-invariant norm $||\cdot||_{\ffg^*}$ on $\ffg^*$ given by
\begin{gather*}
||\rho||_{\ffg^*}:=\sqrt{(\rho,\rho)_{\ffg^*}}
\end{gather*}
for all $\rho\in\ffg^*$. The inner product on $\ffg$ also induces a $G$-invariant sup norm $||\cdot||_{\infty}$ on the dual~$\ffg^*$ given by
\begin{gather*}
||\rho||_{\infty}:=\sup_{||\nu||_{\ffg}=1} |\langle\rho,\nu\rangle|
\end{gather*}
for all $\rho\in\ffg^*$, where $\langle\cdot,\cdot\rangle\colon \ffg^*\times\ffg\to\bbr$ is the pairing between $\ffg$ and $\ffg^*$. Recall that the sup norm satisfies $|\langle\rho,\nu\rangle|\le ||\rho||_{\infty}||\nu||_{\ffg}$ for all $\rho\in\ffg^*$ and all $\nu\in\ffg$.
\end{Remark}

We recall the following definition:

\begin{Definition}[constants of motion]\label{Def:Constants}
Let $M$ be a manifold and let $X$ be a smooth vector field on~$M$. A smooth function $f\in C^\infty(M)$ on the manifold $M$ is a {\it constant of motion of~$X$} if $\d f (X)=0$.
\end{Definition}

\begin{Remark}\label{Rem:Constants}
Let $X$ be a smooth vector field on a manifold $M$, and let $f$ be a smooth function on~$M$. A straightforward check shows that $f$ is a constant of motion of $X$ if and only if $f$ is conserved along each integral curve of~$X$. Here we say $f$ is conserved, or is constant, along an integral curve $\alpha$ of~$X$, if for every time $t$ such that $\alpha$ is defined we have that $f(\alpha(t))=f(\alpha(0))$.
\end{Remark}

\begin{Remark}\label{Rem:HamConstants} Recall that if $(M,\omega,G,\Phi,h)$ is a Hamiltonian $G$-system, then the Hamiltonian~$h$ is a constant of motion of the Hamiltonian vector field~$\Xi_h$. Furthermore, it is Noether's theo\-rem that the momentum~$\Phi$ is constant along the integral curves of $\Xi_h$ (see, for example, \cite[Theorem~2.2]{M92}). That is, for all vectors $\xi\in\ffg$, the function $\langle\Phi,\xi\rangle$ is a constant of motion of the Hamiltonian vector field~$\Xi_h$. Suppose we are given an inner product on the Lie algebra $\ffg$ of $G$, and let $||\cdot||_{\ffg^*}$ be the dual norm on the dual~$\ffg^*$. Then the function $||\Phi||^2_{\ffg^*}$ is also a constant of motion of the Hamiltonian vector field $\Xi_h$.
\end{Remark}

We will make use of the following relationship between the invariant constants of motion of isomorphic vector fields:

\begin{Lemma}\label{Lemma:Constants} Let $M$ be a proper $G$-manifold, let $X$ and $Y$ be smooth vector fields on $M$, and let $f$ be a $G$-invariant constant of motion of~$X$. Let $H$ be a Lie subgroup of $G$ with Lie algebra~$\ffh$, and suppose that there exists an infinitesimal gauge transformation $\psi\in C^\infty(M,\ffh)^H$ such that $Y=X+\psi_M$. Then $f$ is a constant of motion of the vector field~$Y$. In particular, if $X$ and $Y$ are $G$-isomorphic vector fields, then they have the same $G$-invariant constants of motion.
\end{Lemma}

\begin{proof}We verify that $\d f (Y) = 0$. Since $f$ is a constant of motion of $X$, observe that
\begin{gather*}
\d f (Y) = \d f (X+\psi_M)= \d f (X) + \d f (\psi_M)= \d f (\psi_M).
\end{gather*}
Hence, it suffices to show that $\d f (\psi_M)=0$. Let $m$ be a point in $M$, then
\begin{align*}
(\d f)_m (\psi_M(m))&=(\d f)_m \left(\frac{\d}{\d t}\Big|_0\exp(t\psi(m))\cdot m\right) \\
&=\frac{\d}{\d t}\Big|_0 f\big(\exp(t\psi(m))\cdot m\big) \\
&=\frac{\d }{\d t}\Big|_0 f(m) \qquad \text{since }f\text{ is }G\text{-invariant} \\
&=0.
\end{align*}
Since the point $m$ is arbitrary, we have that $\d f (\psi_M)=0$. Hence, $f$ is a constant of motion of~$Y$.
\end{proof}

Lemma \ref{Lemma:Constants} has the following two corollaries:

\begin{Corollary}\label{Cor:ConstantsAugHam}Let $(M,\omega,G,\Phi,h)$ be a Hamiltonian $G$-system, let $\eta$ be a vector in the Lie algebra $\ffg$ of $G$, and let $\Xi_{h^\eta}$ be the Hamiltonian vector field augmented by~$\eta$. Suppose $||\cdot||_{\ffg^*}$ is the $G$-invariant norm on the dual $\ffg^*$ that is dual to a given $G$-invariant inner product on the Lie algebra $\ffg$ of~$G$. Then the Hamiltonian function $h$ and the squared-norm of the momentum~$||\Phi||^2_{\ffg^*}$ are constants of motion of the augmented Hamiltonian vector field $\Xi_{h^\eta}$.
\end{Corollary}

\begin{proof} We want to apply Lemma~\ref{Lemma:Constants}. Observe that the functions $h$ and $||\Phi||^2_{\ffg^*}$ are $G$-invariant constants of motion of the Hamiltonian vector field $\Xi_h$ (Remark~\ref{Rem:HamConstants}). Let $H$ be the Lie subgroup $G_\eta:=\{g\in G\,|\, \Ad(g)\eta=\eta\}$, and let $\ffh$ be its Lie algebra. Let $\psi\colon M\to\ffh$ be the constant map $\psi=-\eta$. Then $\psi$ is an infinitesimal gauge transformation in $C^\infty(M,\ffh)^H$ such that $\Xi_{h^\eta}=\Xi_h+\psi_M$. Hence, we obtain the result by Lemma~\ref{Lemma:Constants} with $\Xi_h$ in place of~$X$ and~$\Xi_{h^\eta}$ in place of~$Y$.
\end{proof}

\begin{Corollary}\label{Cor:ConstantsVerticalAugHam}
Let $G$ be a Lie group with Lie algebra $\ffg$, let $K$ be a Lie subgroup of $G$ with Lie algebra $\ffk$, let $V$ be a $K$-manifold, and let $M$ be the associated bundle $G\times^K V$. Suppose that $(M,\omega,G,\Phi,h)$ is a Hamiltonian $G$-system.
Furthermore:
\begin{enumerate}\itemsep=0pt
\item[$1.$] Let $\ffg=\ffk\oplus\ffq$ be a $K$-equivariant splitting, and let $\nu\colon \Gamma(TM)\to\Gamma(TV)$ be the vertical projection of vector fields with respect to this splitting $($Definition {\rm \ref{Def:SplitConnProj})}.
\item[$2.$] Let $||\cdot||_{\ffg^*}$ be the $G$-invariant norm on the dual $\ffg^*$ that is dual to a given $G$-invariant inner product on the Lie algebra $\ffg$ of $G$.
\item[$3.$] Let $\eta$ be a vector in the Lie algebra $\ffk$, and let $\Xi_{h^\eta}$ be the Hamiltonian vector field augmented by $\eta$.
\end{enumerate}
Then the Hamiltonian function $h$ and the squared-norm of the momentum $||\Phi||^2_{\ffg^*}$ are constants of motion of the vertical projection~$\nu(\Xi_{h^\eta})$.
\end{Corollary}

\begin{proof} By Corollary \ref{Cor:ConstantsAugHam}, the functions $h$ and $||\Phi||^2_{\ffg^*}$ are $G$-invariant constants of motion of the augmented Hamiltonian vector field $\Xi_{h^{\eta}}$ (Remark~\ref{Rem:HamConstants}). We want to apply Lemma \ref{Lemma:Constants}. Observe that the group $K$, acting on $M$ as a subgroup of the Lie group~$G$, is the stabilizer of the point $m:=[1,0]\in M$. Furthermore, the $K$-manifold $V$, together with the $K$-equivariant embedding $j\colon V\hookrightarrow M$ defined by $j(w):=[1,w]$, is a global slice of the action through the point~$m$ (Remark \ref{Rem:ArbitraryModel}). Let $E\colon \bbx(K\times V \toto V)\to\bbx(G\times M \toto M)$ be the canonical equivariant extension functor (Definition \ref{Def:ExtensionFunctor}) and let $P\colon \bbx(G\times M \toto M)\to\bbx(K\times V \toto V)$ be the projection functor corresponding to the splitting $\ffg=\ffk\oplus\ffq$ (Definition \ref{Def:ProjectionFunctor}). By Theorem~\ref{Thm:LermanEquivalence}, there exists a map $\psi\in C^\infty(M,\ffg)^G$ such that
\begin{gather}\label{Eq:GaugeTransf}
\Xi_h=E(P(\Xi_h))+\psi_M.
\end{gather}
By definition of the functors $E$ and $P$, the vector field $\Xi_h$ is such that $E(P(\Xi_h))=\nu(\Xi_h)$. Since the vector $\eta$ is in $\ffk$, the vector field $\eta_M$ is vertical in the bundle $M\to G/K$, and thus $\nu(\eta_M)=\eta_M$. Consequently, we have that
\begin{align*}%\label{Eq:VertTransf}
\begin{split}
\Xi_{h^\eta}
&=\Xi_h-\eta_M \\
&=E(P(\Xi_h))+\psi_M-\eta_M \\
&=\nu(\Xi_h)+\psi_M-\nu(\eta_M)\qquad\text{by (\ref{Eq:GaugeTransf})}\\
&=\nu(\Xi_{h^\eta})+\psi_M.
\end{split}
\end{align*}
That is, the same map $\psi$ relates the augmented Hamiltonian vector field $\Xi_{h^\eta}$ and its vertical part. As a side note, observe that this doesn't mean the vector field $\Xi_{h^\eta}$ is isomorphic to its vertical part since neither vector field is necessarily $G$-equivariant; and while both vector fields are equivariant with respect to the subgroup $G_\eta:=\{g\in G \,|\, \Ad(g)\eta=\eta\}$, the map $\psi$ doesn't necessarily land in the Lie algebra of~$G_\eta$. However, applying Lemma~\ref{Lemma:Constants} with the vector field $\Xi_{h^\eta}$ in place of $X$, the vector field $\nu\left(\Xi_{h^\eta}\right)$ in place of $Y$, and the infinitesimal gauge transformation $-\psi\in C^\infty(M,\ffg)^G$, one obtains that the functions $h$ and $||\Phi||^2_{\ffg^*}$ are constants of motion of the vector field $\nu (\Xi_{h^\eta} )$.
\end{proof}

We will need the following application of the Morse lemma for families, which generalizes a~similar application in \cite[Proposition~3.4]{LS98}:

\begin{Proposition}\label{Prop:MorseApp}
Let $U$ and $W$ be normed finite-dimensional vector spaces.
Furthermore:
\begin{enumerate}\itemsep=0pt
\item[$1.$] Let $f\in C^\infty(U\times W)$ be a smooth function such that $f(0,0)=0$, $\d_Wf(0,0)=0$, and $\d^2_Wf(0,0)$ is nondegenerate and positive definite, where $\d_W$ and $\d^2_W$ denote the differential and Hessian in the $W$ variables respectively.
\item[$2.$] Let $\varphi\in C^\infty(U\times W)$ be a smooth function such that $\varphi(0,0)=0$ and $\varphi(\rho,w)\ge |\rho|^2$ for all $(\rho,w)\in U\times W$.
\item[$3.$] Let $\theta\in C^0(U\times W)$ be a nonnegative continuous function with $\theta(0,0)=0$.
\end{enumerate}
Let the norm on the product space $U\times W$ be the sum of the squares of the norms on $U$ and $W$.
Then for all $\epsilon >0$, there exists $\delta>0$ such that if $\gamma(t)=(\rho(t),w(t))$ defines a curve in $U\times W$ with
\begin{enumerate}\itemsep=0pt
\item[$1)$] $|\gamma(0)|<\delta$;
\item[$2)$] $\varphi(\gamma(t))\le\varphi(\gamma(0))$ for all $t$;
\item[$3)$] $|f(\gamma(t))|\le \theta(\gamma(0))$ for all $t$.
\end{enumerate}
Then $|\gamma(t)|<\epsilon$ for all $t$.
\end{Proposition}

\begin{proof}
Since $0$ is a critical point of $f(0,\cdot)$ and $\d^2_Wf(0,0)$ is nondegenerate, by the Morse lemma for families \cite[Lemma~1.2.2]{D96}, there exists a neighborhood $\widetilde U$ of $0$ in~$U$, a neighborhood $\widetilde W$ of $0$ in $W$, and a smooth map $\sigma\colon \widetilde U\rightarrow \widetilde W$ implicitly defined by the equation
\begin{gather*}
\d_Wf(\rho,\sigma(\rho))=0
\end{gather*}
for all $\rho\in \widetilde U$.
Additionally, there is also a smooth map $y\colon \widetilde U\times \widetilde W\rightarrow W$ implicitly defined by the equation
\begin{gather}\label{morse1}
f(\rho,v)=f(\rho,\sigma(\rho))+\frac{1}{2}\d^2_Wf(\rho,\sigma(\rho))\big(y(\rho,w),y(\rho,w)\big)
\end{gather}
for all $(\rho,w)\in \widetilde U\times \widetilde W$.

Now let $\epsilon>0$ be given. We can always choose $\epsilon ' > 0$ such that if $|\rho|^2<\epsilon'$ and $|y(\rho,w)|^2<\epsilon'$ then $|(\rho,w)|=|\rho|^2+|w|^2<\epsilon$. Thus, if $\gamma(t)=(\rho(t),w(t))$ is a curve in $U\times W$ satisfying the hypotheses, it suffices to show that there exists a $\delta>0$ such that if $|\gamma(0)|<\delta$ then
\begin{gather*}
|\rho(t)|^2<\epsilon' \qquad\text{and} \qquad |y(\rho(t),w(t))|^2<\epsilon'.
\end{gather*}
Since $\d^2_Wf(0,0)$ is positive definite, there exists $\beta_1>0$ and $C>0$ such that if $|\rho|^2<\beta_1$ then
\begin{gather}\label{morse2}
C|w|^2<\d^2_Wf(\rho,\sigma(\rho))(w,w)
\end{gather}
for every $w\in W$. By the continuity of $f$ and $\sigma$ there exists $\beta_2>0$ such that if $|\rho|^2<\beta_2$ then
\begin{gather}\label{morse3}
|f(\rho,\sigma(\rho))|<\frac{\epsilon'C}{4}.
\end{gather}
By the continuity of $\varphi$ there exists $\alpha_1>0$ such that if $|(\rho,w)|<\alpha_1$ then
\begin{gather}\label{morse4}
|\varphi(\rho,w)|=\varphi(\rho,w)<\min (\beta_1,\beta_2,\epsilon ').
\end{gather}
By the continuity of $\theta$ there exists $\alpha_2>0$ such that if $|(\rho,w)|<\alpha_2$ then
\begin{gather}\label{morse5}
|\theta(\rho,w)|<\frac{\epsilon'C}{4}.
\end{gather}
Now set $\delta:=\min (\alpha_1,\alpha_2)$ then note that for $\gamma(t)=(\rho(t),w(t))$ satisfying the hypotheses of the proposition
\begin{align*}
|\rho(t)|^2&\le\varphi(\rho(t),w(t)) \\
&\le\varphi(\rho(0),w(0)) \\
&<\min (\beta_1,\beta_2,\epsilon')\\
&\le \epsilon',
\end{align*}
where we use inequality (\ref{morse4}).
This was the first inequality we needed to prove.

Now, since $|\rho(t)|^2<\beta_1$ for all $t$, inequality (\ref{morse2}) gives
\begin{gather}\label{morse6}
C|w|^2<\d^2_W f \big( \rho(t),\sigma(\rho(t)) \big)(w,w)
\end{gather}
for all $w\in W$ and all $t$. Similarly, since $|\rho(t)|^2<\beta_2$ for all $t$, inequality (\ref{morse3}) gives
\begin{gather}\label{morse7}
\big|f\big(\rho(t),\sigma(\rho(t))\big)\big|<\frac{\epsilon'C}{4}
\end{gather}
for all $t$. Finally, note that
\begin{align*}
\big|y\big(\rho(t),\sigma(\rho(t))\big)\big|^2
& \le \frac{1}{C} \d^2_W f \big( \rho(t),\sigma(\rho(t)) \big) \big(y(\rho(t),\sigma(t)),y(\rho(t),\sigma(t))\big)
\qquad \text{by (\ref{morse6})} \\
& \le \frac{2}{C}\big( \big|f\big(\rho(t),w(t)\big)\big|+\big|f\big(\rho(t),\sigma(\rho(t))\big)\big|\big)
\qquad \text{ by (\ref{morse1})}\\
& \le \frac{2}{C}\big( \theta\big(\rho(0),w(0)\big)+\big|f\big(\rho(t),\sigma(\rho(t))\big)\big|\big)
\qquad \text{by assumption (3) on }\gamma\\
&<\frac{2}{C}\left(\frac{\epsilon ' C}{4}+\frac{\epsilon ' C}{4}\right)
\qquad \text{ by (\ref{morse5}) \& (\ref{morse7})}
\\
&=\epsilon '.
\end{align*}
This was the second inequality we needed to prove. Thus, we have that $|\gamma(0)|<\delta$ implies $|\gamma(t)|<\epsilon$ for all $t$ as desired.
\end{proof}

We will also need the following standard construction:

\begin{Definition}\label{Def:SymplecticSlice}
Let $(M,\omega,G,\Phi)$ be a Hamiltonian $G$-space and let $m\in M$ be a point in the manifold $M$. The {\it symplectic slice} at the point $m$ is the vector space
\begin{gather*}
W:=\ker\d\Phi_m/\big(T_m(G\cdot m)\cap\ker\d\Phi_m\big).
\end{gather*}
\end{Definition}

\begin{Remark}\label{Rem:SymplecticSlice}
Let $(M,\omega,G,\Phi)$ be a Hamiltonian $G$-space and let $W$ be the symplectic slice at a point $m\in M$. The symplectic slice inherits a canonical Hamiltonian representation of the stabilizer subgroup $G_m$ (see, for example, \cite[Theorem~7.1.1(iii)]{OR04}). Furthermore, let $\mu$ be the moment of the point $m$ and observe that $T_m(G\cdot m)\cap\ker\d\Phi_m=T_m(G_\mu\cdot m)$. Thus, if the moment is fixed by the coadjoint representation, then the symplectic slice takes the form $W=\ker\d\Phi_m/T_m(G\cdot m)$.
\end{Remark}

We recall the standard Marle--Guillemin--Sternberg normal form:

\begin{Theorem}[MGS normal form \cite{GuS84, Ma85}]\label{Thm:MGSNormalForm}
Let $(M,\omega,G,\Phi)$ be a Hamiltonian $G$-space. Let \smash{$m\in M$} be a point in the manifold with moment $\Phi(m)=0$ and suppose the moment isotro\-py~$G_{\Phi(m)}$ acts properly on the manifold. Let $K$ be the stabilizer of the point $m\in M$, let $\ffk$ be the corresponding Lie algebra, let $\ffk^0$ be the annihilator of $\ffk$, and let $W$ be the symplectic slice at the point~$m$ $($Definition~{\rm \ref{Def:SymplecticSlice})}.
Given a $K$-equivariant embedding $\iota\colon \ffk^*\hookrightarrow\ffg^*$, there exist $K$-invariant open neighborhoods $\ffk_r^0\subseteq \ffk^0$ and $W_r\subseteq W$ of the origins in the respective vector spaces~$\ffk^0$ and~$W$, a $G$-invariant symplectic form $\omega_Z$, and a~$G$-equivariant map $\Phi_Z\colon Z\to\ffg^*$ on the associated bundle
\begin{gather*}
Z:=G\times^K\big(\ffk^0_r\times W_r\big),
\end{gather*}
such that the quadruple $(Z,\omega_z,G,\Phi_Z)$ is a Hamiltonian $G$-space.
Furthermore:
\begin{enumerate}\itemsep=0pt
\item[$1.$] There exists a $G$-invariant neighborhood $U_m\subseteq M$ of the orbit $G\cdot m$, and a $G$-equivariant symplectomorphism $\Psi\colon Z\to U_m$ such that $\Psi^*\Phi=\Phi_Z$.
\item[$2.$] The momentum map $\Phi_Z\colon Z\to\ffg^*$ is given by
\begin{gather*}
\Phi_Z ([g,\rho,w] )=\CoAd(g)\big(\rho+\iota(\Phi_W(w))\big),
\end{gather*}
where $\Phi_W\colon W\to\ffk^*$ is the momentum map for the representation of the stabilizer $K$ on the symplectic slice~$W$.
\end{enumerate}
\end{Theorem}

\begin{Definition}\label{Def:MGSNormalForm}
Let $(M,\omega,G,\Phi)$ be a Hamiltonian $G$-space and let $m\in M$ be a point with moment $\Phi(m)=0$. The {\it Marle--Guillemin--Sternberg normal form} (MGS normal form) is the Hamiltonian $G$-space $\big(Z:=G\times^K\big(\ffk_r^0\times W_r\big),\omega_Z, G, \Phi_Z\big)$ of Theorem~\ref{Thm:MGSNormalForm}. For a Hamiltonian $G$-system $(M,\omega,G,\Phi,h)$ with a point $m\in M$ with moment $\Phi(m)=0$, the MGS normal form is the Hamiltonian $G$-system $(Z,\omega_Z, G, \Phi_Z,h_Z)$ where $(Z,\omega_Z, G, \Phi_Z)$ is the MGS normal form of the underlying Hamiltonian $G$-space and the Hamiltonian function $h_Z$ is the pullback of the Hamiltonian $h$ by the symplectomorphism $\Psi$ in part (1) of Theorem~\ref{Thm:MGSNormalForm}.
\end{Definition}

\begin{Remark}\label{Rem:OnMGSForm}
Let $(M,\omega,G,\Phi)$ be a Hamiltonian $G$-space.
Let $m\in M$ be a point with moment $\Phi(m)=0$, and such that the moment isotropy $G_{\Phi(m)}$ acts properly on $M$.
Let $K$ be the stabilizer of $m$, and let $\ffk$ and $\ffk^0$ be the Lie algebra of $K$ and the annihilator of $\ffk$ respectively.
Then:
\begin{enumerate}\itemsep=0pt
\item Since $\Phi(m)=0$, the moment of $m$ is fixed by the coadjoint representation of the Lie group $G$.
Hence, $G_{\Phi(m)}=G$ and thus the Lie group $G$ acts properly on the manifold $M$.
In this case, a $G_m$-equivariant embedding $\iota\colon \ffk^*\hookrightarrow \ffg$ is guaranteed to exist (see \cite[Remark~3.2]{LS98}).
\item Let $(Z,\omega_Z,G,\Phi_Z)$ be the MGS normal form for the Hamiltonian $G$-space $(M,\omega,G,\Phi)$ at the point~$m$. The manifold $V:=\ffk_r^0\times W_r\subseteq \ffk^0\times W$ is a global slice for the action of~$G$ on the manifold~$Z$ through the point $[1,0,0]$ (see Definition~\ref{Def:Slices}). The corresponding $K$-equivariant embedding $j\colon V\hookrightarrow Z$ is defined by $j(\rho,w):=[1,\rho,w]$.
\item One can construct the MGS normal form without requiring that the moment $\Phi(m)$ be fixed by the coadjoint representation (see, for example, \cite[Theorem~7.5.5]{OR04}). However, when it is fixed, the MGS normal form, the expression for the moment map, and some of our arguments become much simpler.
\end{enumerate}
\end{Remark}

We can now prove the following, which is Theorem~\ref{Thm:MROTheorem} for the special case where the Hamiltonian $G$-system is in MGS normal form and we've made the assumptions in Remark~\ref{Rem:FixedMomentum}. In Section~\ref{MGSNormalForm}, we reduce Theorem~\ref{Thm:MROTheorem} to this special case.

\begin{Theorem}\label{Thm:MainTheorem2} Let $G$ be a Lie group with Lie algebra $\ffg$, let $K$ be a compact Lie subgroup of $G$ with Lie algebra $\ffk$, let $V$ be a $K$-manifold, and let $Z$ be the associated bundle $G\times^{K}V$. Suppose that $(Z,\omega,G,\Phi,h)$ is a~Hamiltonian $G$-system. Furthermore:
\begin{enumerate}\itemsep=0pt
\item[$1.$] Let $\ffk_r^0$ be a $K$-invariant neighborhood of the origin in the annihilator $\ffk^0$ of the Lie algebra~$\ffk$.
\item[$2.$] Let $W_r$ be a $K$-invariant neighborhood of the origin in a symplectic vector space~$W$, where~$W$ has a Hamiltonian representation of the group $K$ and corresponding equivariant momentum map $\Phi_W\colon W\to \ffk^*$.
\item[$3.$] Suppose that the $K$-manifold $V$ is the product $\ffk^0_r\times W_r$, and let the $K$-equivariant embedding $j\colon V\hookrightarrow Z$ be defined by $j(\rho,w):=[1,\rho,w]$.
\item[$4.$] Suppose there exists an $\Ad(G)$-invariant inner product on the Lie algebra $\ffg$, let $\ffq$ be the orthogonal complement to $\ffk$ in $\ffg$ with respect to this inner product, and let $\bbp\colon \ffg\to\ffk$ be the corresponding projection;
\item[$5.$] Let the map $\iota\colon \ffk^*\hookrightarrow\ffg^*$ be the $K$-equivariant embedding defined for every $\rho\in\ffk^*$ by
\begin{gather*}
\iota(\rho)\colon \ \ffg\to\bbr, \qquad \iota(\rho)(\eta):=\rho(\bbp(\eta)).
\end{gather*}
\item[$6.$] For all $[g,\rho,w]\in Z$, let the momentum map $\Phi\colon Z\to\ffg^*$ be given by
\begin{gather*}
\Phi([g,\rho,w])=\CoAd(g)\big(\rho+\iota(\Phi_W(w))\big).
\end{gather*}
\item[$7.$] Suppose that the point $m:=[1,0,0]\in Z$ is a $G$-relative equilibrium of the Hamiltonian vector field~$\Xi_h$, and suppose it is such that $h(m)=0$ and $\Phi(m)=0$.
\end{enumerate}
If there exists a velocity $\xi\in\ffg$ such that the Hessian $\d^2h^\xi(m)$ is definite and nondegenerate on the subspace $T_mj (\{0\}\times W_r )\subseteq T_mZ$, where $h^\xi:=h-\langle\Phi,\xi\rangle$ is the augmented Hamiltonian, then the relative equilibrium~$m$ is $G$-stable.
\end{Theorem}

\begin{proof}
We want to prove the $G$-stability of the $G$-relative equilibrium $m=[1,0,0]\in Z$ of the Hamiltonian vector field $\Xi_h$ on the manifold $Z$.
Observe that the stabilizer of the point $m$ is the Lie subgroup $K$ acting on~$Z$ as a subgroup of the group~$G$. Also observe that the $K$-manifold $V$, together with the $K$-equivariant embedding $j\colon V\hookrightarrow Z$, is a global slice for the action through the point~$m$ (see Remark~\ref{Rem:OnMGSForm}). Thus, in particular, the manifold $Z$ is the tube $G\cdot V$ generated by the slice $V$. Observe that the complement $\ffq$ gives a $K$-equivariant splitting $\ffg=\ffk\oplus\ffq$, and hence we obtain a projection functor $P\colon \bbx(G\times Z\toto Z)\to\bbx(K\times V\toto V)$ with respect to the splitting (Definition~\ref{Def:ProjectionFunctor}). The proof is an application of the slice stability criterion (Theo\-rem~\ref{Thm:MainTheorem1}). Thus, we need to give a Lie subgroup $H$ of the stabilizer $K$, and a vector field $Y$ on the slice $V$ that is $H$-isomorphic to the projected vector field $P(\Xi_h)$. Furthermore, we need to show that the point $\vzero=j^{-1}(m)$ is $H$-stable for the vector field $Y$.

Use the $K$-equivariant projection $\bbp\colon \ffg\to\ffk$ to define the vector $\eta:=\bbp(\xi)$, which is the projection of the velocity onto the isotropy Lie algebra. Now consider the Lie subgroup $G_\eta:=\{g\in G \,|\, \Ad(g)\eta=\eta\}$ and the Lie subgroup
\begin{gather*}
H:=(G_\eta)_m\equiv(G_m)_\eta\equiv\{g\in G_m\,|\,\Ad(g)\eta=\eta\}.
\end{gather*}
The subgroup $H$ is the Lie subgroup of the stabilizer that we will use to apply the slice stability criterion.

The inner product on $\ffg$ also determines an $H$-equivariant splitting $\ffg_\eta=\ffh\oplus\ffp$ of the Lie algebra $\ffg_\eta$ of the subgroup~$G_\eta$. Let $P_{G_\eta}\colon \bbx(G_\eta\times Z \toto Z)\to\bbx(H\times V\toto V)$ be the projection functor of $G_\eta$-equivariant vector fields with respect to this splitting of $\ffg_\eta$ (Definition~\ref{Def:ProjectionFunctorSubgroup}). Let $h^\eta:=h-\langle\Phi,\eta\rangle$ be the Hamiltonian augmented by the vector $\eta$ and let $\Xi_{h^\eta}=\Xi_h-\eta_Z$ be the corresponding augmented Hamiltonian vector field (Definition~\ref{Def:AugHam}). Take the vector field~$Y$, required for the application of the slice stability criterion, to be the projected vector field $Y:=P_{G_\eta}(\Xi_{h^\eta})$. Note that the Hamiltonian vector field $\Xi_h$ is in particular $G_\eta$-equivariant and, by Lemma~\ref{Lemma:AugIso}, it is $G_\eta$-isomorphic to the vector field $\Xi_{h^\eta}$. Thus, by the functoriality of $P_{G_\eta}$, the vector field~$Y$ is $H$-isomorphic to the vector field~$P_{G_\eta}(\Xi_h)$. Observe that $P_{G_\eta}(\Xi_h)=P(\Xi_h)$, so the vector fields $Y$ and $P(\Xi_h)$ are $H$-isomorphic.

It remains to show that the origin $j^{-1}(m)$ is $H$-stable for the vector field $Y$. By Lemma~\ref{Lemma:SubgroupStability}, it is enough to show that the origin $j^{-1}(m)$ is $\{1\}$-stable for~$Y$. That is, we show that for any $\epsilon>0$, there exists a $\delta>0$, with $0<\delta<\epsilon$, such that all maximal integral curves of the vector field $Y$ starting at points in the $\delta$-ball around the origin stay in the $\epsilon$-ball around the origin for all times for which they are defined. Proposition~\ref{Prop:MorseApp} guarantees we can obtain such a~$\delta$. In particular, we apply Proposition~\ref{Prop:MorseApp} with the vector space~$\ffk^0$ in place of~$U$ and the vector space~$W$ of this theorem in place of the~$W$ in the proposition.

In order to apply Proposition \ref{Prop:MorseApp} we need to pick norms on the spaces $\ffk^0$ and~$W$. For this, let $||\cdot||_\ffg$ be the $G$-invariant norm on the Lie algebra induced by the given inner product on~$\ffg$, let $||\cdot||_{\ffg^*}$ be the dual norm induced by the dual inner product on~$\ffg^*$, and let $||\cdot||_{\infty}$ be the corresponding $G$-invariant sup norm on the dual~$\ffg^*$ (see Remark~\ref{Rem:NormsOnLie}). The norm $||\cdot||_{\ffg^*}$ gives an invariant norm on the subspace~$\ffk^0$ of~$\ffg^*$, pick this norm for~$\ffk^0$. Pick any norm on the vector space~$W$. This works since the slice~$V$ is contained in the vector space $\ffk^0\times W$. Since all norms on a finite-dimensional vector space are equivalent, there exists a constant $A>0$ such that
\begin{gather}\label{Eq:EquivalentNorms}
||\cdot||_{\infty} \le A||\cdot ||_{\ffg^*}.
\end{gather}
We apply Proposition \ref{Prop:MorseApp} with the functions
\begin{align}\label{Eq:MorseAppFunctions}
\begin{split}
& f:= j^*h^\eta \in C^\infty(V), \\
& \varphi:=||j^*\Phi||^2_{\ffg^*}\in C^\infty(V), \\
& \theta:=|j^*h^\eta|+A||j^*\Phi||_{\ffg^*}||\eta||_{\ffg}\in C^0(V).
\end{split}
\end{align}
To complete the proof we need to show that the functions in~(\ref{Eq:MorseAppFunctions}) and the maximal integral curves of the vector field $Y$ satisfy the hypotheses of Proposition~\ref{Prop:MorseApp}.

Now we verify that the functions (\ref{Eq:MorseAppFunctions}) satisfy the conditions of Proposition~\ref{Prop:MorseApp}. For the function $f$ we need to show that $f(\vzero)=0$, $\d_Wf(\vzero)=0$, and that the Hessian $\d^2_Wf(\vzero)$ is definite and nondegenerate. Here $\d_Wf(\vzero)$ and $\d^2_Wf(\vzero)$ are respectively the differential and the Hessian of the function $f$ in the $W$ variables. For the first, note that
\begin{gather*}
f(\vzero)
=h^\eta(j(\vzero))
=h^\eta(m)
=h(m)-\langle\Phi(m),\eta\rangle
=0-\langle 0, \eta\rangle
=0.
\end{gather*}
Now note that the image of the dual $\ffk^*$ under the $K$-equivariant embedding $\iota\colon \ffk^*\hookrightarrow\ffg^*$ is the annihilator $\ffq^0$.
Thus we have that
\begin{gather}\label{Eq:MomentImage}
\iota(\Phi_W(w))\in \ffq^0
\end{gather}
for all $w\in W$.
Let $\xi^\perp\in\ffq$ be the vector in $\ffq$ such that $\xi=\eta+\xi^\perp$. Then, using (\ref{Eq:MomentImage}) and the explicit form of the momentum map, we have that
\begin{gather*}
\langle\Phi,\xi\rangle([1,0,w]) =\langle\Phi,\eta\rangle([1,0,w])+\big\langle\iota(\Phi_W(w)),\xi^\perp\big\rangle =\langle\Phi,\eta\rangle([1,0,w])
\end{gather*}
for all $w\in W$. This in turn implies that $h^\xi|_{j(W_r)}=h^\eta|_{j(W_r)}$, where we identify the spaces $W_r$ and $\{0\}\times W_r\subseteq V$.
Hence,
\begin{gather*}
0=\d h^\xi(m)|_{T_mj(W_r)} =\d \big(h^\xi|_{j(W_r)}\big)(\vzero)
=\d \big(h^\eta|_{j(W_r)}\big)(\vzero) =\d \big(j^*h^\eta|_{W_r}\big)(\vzero)=\d_Wf(\vzero),
\end{gather*}
which is the second condition we needed to check. This condition guarantees that the Hessian $\d^2_Wf(\vzero)$ is well-defined. The equality $h^\xi|_{j(W_r)}=h^\eta|_{j(W_r)}$ and the pullback properties of Hessians (see Remark~\ref{Rem:Hessians}) imply that
\begin{gather*}
\d^2h^\xi(m)|_{T_mj(W_r)} =\d^2\big(h^\xi|_{j(W_r)}\big)(\vzero)=\d^2\big(h^\eta|_{j(W_r)}\big)(\vzero)=\d^2\big(j^*h^\eta|_{W_r}\big)(\vzero)
=\d^2_Wf(\vzero).
\end{gather*}
Hence, the Hessian $\d^2_Wf(\vzero)$ is definite and nondegenerate since $\d^2h^\xi(m)|_{T_mj(W_r)}$ is such by assumption. It is of no loss of generality to assume that the Hessian $\d^2h^\xi(m)|_{T_mj(W_r)}$ is positive definite. Hence, the Hessian $\d^2_Wf(\vzero)$ is positive definite too. This was the third and last thing we needed to verify that $f$ satisfies.

For the function $\varphi$ we need to show that $\varphi(\vzero)=0$ and that $\varphi(\rho,w)\ge ||\rho||^2_{\ffg^*}$ for all $(\rho,w)\in V$. For the first, note
\begin{gather*}
\varphi(\vzero)=||j^*\Phi||^2_{\ffg^*}(\vzero)=||\Phi(m)||^2_{\ffg^*}=||0||^2_{\ffg^*}=0.
\end{gather*}
For the second, note that for all $(\rho,w)\in V$ we have that $\rho\in\ffk^0$ and $\iota(\Phi_W(w))\in\ffq^0=(\ffk^\perp)^0=(\ffk^0)^\perp$. Thus, for all $(\rho,w)\in V$ we have that
\begin{gather*}%\label{Eq:MomentEstimate}
\varphi(\rho,w)=||\rho+\iota(\Phi_W(w))||_{\ffg^*}^2=||\rho||_{\ffg^*}^2+||\iota(\Phi_W(w))||_{\ffg^*}^2\ge ||\rho||_{\ffg^*}^2.
\end{gather*}
Hence, $\varphi$ satisfies the desired conditions.

For the function $\theta=|j^*h^\eta|+A||j^*\Phi||_{\ffg^*}||\eta||_\ffg$ we need to show that $\theta(\vzero)=0$. Using that $h(m)=0$ and $\Phi(m)=0$, note we have that
\begin{align*}
\theta(\vzero)
&=|j^*h^\eta(\vzero)|+A||j^*\Phi(\vzero)||_{\ffg^*}||\eta||_\ffg\\
&=|h(m)-\langle\Phi(m),\eta\rangle|+A||\Phi(m)||_{\ffg^*}||\eta||_\ffg\\
&=|h(0)-\langle0,\eta\rangle|+A||0||_{\ffg^*}||\eta||_\ffg\\
&=0.
\end{align*}
Hence, the function $\theta$ satisfies the desired condition.

We now verify the other set of hypotheses of Proposition \ref{Prop:MorseApp}.
Let $\beta$ be an arbitrary maximal integral curve of the vector field $Y$.
Then we need to show that
\begin{align}\label{Eq:Hypotheses1}
\begin{split}
\varphi(\beta(t)) \le \varphi(\beta(0))
\qquad\text{ and }\qquad
|f(\beta(t))| \le \theta(\beta(0))
\end{split}
\end{align}
for all times $t$ for which the curve $\beta$ is defined.

In place of the first inequality in (\ref{Eq:Hypotheses1}), we actually prove the equality $\varphi(\beta(t))=\varphi(\beta(0))$ for all times $t$ such that the integral curve $\beta$ is defined.
Let $\nu\colon \Gamma(TZ)\to\Gamma(\calv Z)$ be the vertical projection with respect to the splitting $\ffg=\ffk\oplus\ffq$ (Definition \ref{Def:SplitConnProj}).
Since the vector fields $\nu\left(\Xi_{h^\eta}\right)$ and $Y$ are $j$-related (see Definition \ref{Def:ProjectionFunctorSubgroup}), the curve $\alpha:=j\circ\beta$ is a maximal integral curve of the vector field $\nu\left(\Xi_{h^\eta}\right)$.
By Corollary \ref{Cor:ConstantsVerticalAugHam}, the function $||\Phi||^2_{\ffg^*}$ is a constant of motion of the vector field $\nu\left(\Xi_{h^\eta}\right)$.
Therefore, we have that
\begin{gather*}
\varphi(\beta(t))
=||j^*\Phi||^2_{\ffg^*}(\beta(t))
=\underbrace{
||\Phi||^2_{\ffg^*}(\alpha(t))
=||\Phi||^2_{\ffg^*}(\alpha(0))
}_{\text{by Corollary \ref{Cor:ConstantsVerticalAugHam}}}
=\varphi(\beta(0)),
\end{gather*}
which is the desired equality.

We now verify the second inequality.
Recall that the norm $||\cdot||_{\infty}$ satisfies $|\langle\rho,\zeta\rangle|\le ||\rho||_{\infty}||\zeta||_{\ffg}$ for any $\rho\in\ffg^*$ and $\zeta\in\ffg$.
Also, as mentioned above, the curve $\alpha:=j\circ\beta$ is a maximal integral curve of the vector field $\nu\left(\Xi_{h^\eta}\right)$.
Now, observe that
\begin{align*}
|f(\beta(t))|
&=|j^*h^\eta(\beta(t))|\\
&=|h^\eta(\alpha(t))|\\
&=|h(\alpha(t))-\langle\Phi(\alpha(t)),\eta\rangle|\\
&\le|h(\alpha(t))|+|\langle\Phi(\alpha(t)),\eta\rangle|\\
&=|h(\alpha(0))|+|\langle\Phi(\alpha(t)),\eta\rangle| \qquad\text{ by Corollary \ref{Cor:ConstantsVerticalAugHam}}\\
&\le |h(\alpha(0))|+ ||\Phi(\alpha(t))||_{\infty}||\eta||_{\ffg} \qquad\text{by properties of }||\cdot||_{\infty} \\
&\le |h(\alpha(0))|+ A||\Phi(\alpha(t))||_{\ffg^*}||\eta||_{\ffg} \qquad\text{ by (\ref{Eq:EquivalentNorms})}\\
&= |h(\alpha(0))|+ A||\Phi(\alpha(0))||_{\ffg^*}||\eta||_{\ffg} \qquad\text{by Corollary \ref{Cor:ConstantsVerticalAugHam}}\\
&\le |j^*h(\beta(0))|+ A||j^*\Phi(\beta(0))||_{\ffg^*}||\eta||_{\ffg}\\
&=\theta(\beta(0)).
\end{align*}
This was the second inequality we needed to prove.

Thus, any integral curve $\beta$ of the vector field $Y$ satisfies the hypotheses of Proposition \ref{Prop:MorseApp}. Applying Proposition~\ref{Prop:MorseApp}, yields that the origin $j^{-1}(m)$ is $\{1\}$-stable for the vector field $Y$ as required. Consequently, the $G$-relative equilibrium $m$ is $G$-stable for the Hamiltonian vector field~$\Xi_h$.
\end{proof}

\section{Reduction to the Marle--Guillemin--Sternberg normal form}\label{MGSNormalForm}

The goal of this section is to reduce Theorem \ref{Thm:MROTheorem} (Montaldi and Rodriguez-Olmos's criterion) to the special case Theorem~\ref{Thm:MainTheorem2}, which we proved in the previous section. The main result of this section is Lemma~\ref{Lemma:MGSReduction}.

We need the following result from linear algebra:

\begin{Lemma}\label{Lemma:LinearAlgebraResult}
Let $V$ be a finite dimensional vector space and let $T\colon V\times V\to\bbr$ be a bilinear form on it. Let $U$, $W$, and $\widetilde W$ be linear subspaces of $V$ such that
\begin{gather*}
V=U\oplus W=U\oplus\widetilde W.
\end{gather*}
Furthermore, suppose that $U$ is contained in $\ker T^\sharp$, where $T^\sharp$ is the linear map
\begin{gather*}
T^\sharp\colon \ V\to V^*,\qquad T^\sharp(v):=T(v,\cdot).
\end{gather*}
Then the following are true:
\begin{enumerate}\itemsep=0pt
\item[$1.$] If $T|_W$ is nondegenerate, then $T|_{\widetilde W}$ is nondegenerate as well.
\item[$2.$] If $T|_W$ is positive (respectively negative) definite, then $T|_{\widetilde W}$ is positive $($respectively nega\-ti\-ve$)$ definite as well.
\end{enumerate}
\end{Lemma}

\begin{proof} Since we have the splittings $V=U\oplus W=U\oplus\widetilde W$, there exists a linear map $L\colon W\to U$ such that $\widetilde W$ is the graph of $L$; that is, such that
\begin{gather}\label{Eq:Graph}
\text{Graph}(L):=\{w+Lw\,|\, w\in W\}=\widetilde W.
\end{gather}

Now suppose that $T|_W$ is nondegenerate and let $\widetilde w\in \widetilde W$ be a vector such that $T^\sharp (\widetilde w )=0$. We want to show that $\widetilde w = 0$. By~(\ref{Eq:Graph}), there exists a vector $w\in W$ such that $\widetilde w = w+Lw$. Observe that
\begin{gather*}
0=T^\sharp (\widetilde w ) =T^\sharp(w+Lw)=T^\sharp(w)+T^\sharp(Lw)=T^\sharp(w),
\end{gather*}
where the last equality is because $U$ is contained in $\ker T^\sharp$. This in turn means that $w=0$ since $w\in W$ and $T|_W$ is nondegenerate by hypothesis. Consequently, the given vector $\widetilde w$ is $0$, so $T|_{\widetilde W}$ is nondegenerate.

For the second statement, suppose $T|_W$ is positive definite and let $\widetilde w\in\widetilde W$ be an arbitrary nonzero vector in~$\widetilde W$. By~(\ref{Eq:Graph}) and since $\widetilde w$ is nonzero, there exists a nonzero vector $w\in W$ such that $\widetilde w = w + Lw$. Then note
\begin{gather*}
T(\widetilde w, \widetilde w )=T(w+Lw,w+Lw)
=T(w,w)+T(w,Lw)+T(Lw,w)+T(Lw,Lw)\\
\hphantom{T(\widetilde w, \widetilde w )}{} =T(w,w) >0.
\end{gather*}
The third equality follows because $U$ is contained in $\ker T^\sharp$, and also $Lw\in U$. The inequality follows because $T|_W$ is positive definite. Thus, since $\widetilde w\in \widetilde W$ is an arbitrary nonzero vector of~$\widetilde W$, the map $T|_{\widetilde W}$ is positive definite. A completely analogous argument works in the case when~$T|_W$ is negative definite.
\end{proof}

With this we can prove the following:
\begin{Lemma}\label{Lemma:SymplecticSliceIndep}
Let $(M,\omega, G, \Phi, h)$ be a Hamiltonian $G$-system, let the point $m$ in the mani\-fold~$M$ be a relative equilibrium of the Hamiltonian vector field $\Xi_h$ such that $\Phi(m)=0$, and let $K$ be the stabilizer of the point~$m$. Furthermore, let $W$ and $\widetilde W$ be $K$-invariant subspaces of the kernel $\ker\d\Phi_m$ such that
\begin{gather*}
\ker\d\Phi_m=T_m(G\cdot m)\oplus W=T_m(G\cdot m)\oplus \widetilde W.
\end{gather*}
Suppose there exists a velocity $\xi\in\ffg$ of $m$, with augmented Hamiltonian $h^\xi:=h-\langle\Phi,\xi\rangle$, such that the Hessian $\d^2h^\xi(m)|_W$ is definite and nondegenerate.
Then the Hessian $\d^2h^\xi(m)|_{\widetilde W}$ is definite and nondegenerate.
\end{Lemma}

\begin{proof}
By Lemma \ref{Lemma:LinearAlgebraResult}, it suffices to show that
\begin{gather*}
T_m(G\cdot m)\subseteq \ker \big(\d^2h^\xi(m)|_{\ker\d\Phi_m}\big)^\sharp.
\end{gather*}
Let $u\in T_m(G\cdot m)$ be an arbitrary vector. We need to show that the functional:
\begin{gather*}
\big(\d^2h^\xi(m)\big)^\sharp(u)=\d^2h^\xi(m)(u,\cdot)
\end{gather*}
vanishes on the vector space $\ker\d\Phi_m$. Since $T_m(G\cdot m)=\{\eta_M(m)\,|\,\eta\in\ffg\}$, the vector $u$ is of the form $u=\eta_M(m)$ for some $\eta\in\ffg$. Thus, we want to show that for every vector $v\in\ker\d\Phi_m$ we have that
\begin{gather*}
\d^2h^\xi(m) (\eta_M(m),v )=0.
\end{gather*}
For this, recall that for all $g\in G$, the vector $\Ad(g)\xi$ is a velocity for the $G$-relative equilibrium $g\cdot m$, so that $\d h^{\Ad(g)\xi}(g\cdot m)=0$ for all $g\in G$. In particular, for any $s\in \bbr$, setting $g=\exp(s\eta)$ we have
\begin{align}\label{Eq:VelocityFact}
\begin{split}
0&=\d h^{\Ad(\exp(s\eta))\xi}(\exp(s\eta)\cdot m) \\
&=\d h (\exp(s\eta)\cdot m)-\d\langle\Phi,\Ad(\exp(s\eta)\xi)\rangle(\exp(s\eta)\cdot m).
\end{split}
\end{align}
Now let $\gamma\colon (-\epsilon,\epsilon)\times(-\epsilon,\epsilon)\to M$ be a family of smooth curves such that
\begin{enumerate}\itemsep=0pt
\item[1)] $\gamma(0,0)=m$,
\item[2)] $\frac{\partial}{\partial t}\Big|_{t=0}\gamma(s,t)\in T_{\exp(s\eta)\cdot m}M$ for all $s\in(-\epsilon,\epsilon)$,
\item[3)] $\frac{\partial}{\partial s}\Big|_{s=0}\gamma(s,0)=\eta_M(m)$,
\item[4)] $\frac{\partial}{\partial t}\Big|_{t=0}\gamma(0,t)=v$.
\end{enumerate}
Then for all $s\in(-\epsilon,\epsilon)$, use (\ref{Eq:VelocityFact}) to get that
\begin{align}\label{Eq:tDerivative}
\begin{split}
0&=\d h^{\Ad(\exp(s\eta))\xi}(\exp(s\eta)\cdot m)\left(\frac{\partial}{\partial t}\Big|_{t=0}\gamma(s,t)\right) \\
&=\big(\d h (\exp(s\eta)\cdot m)-\d\langle\Phi,\Ad(\exp(s\eta)\xi)\rangle(\exp(s\eta)\cdot m)\big)\left(\frac{\partial}{\partial t}\Big|_{t=0}\gamma(s,t)\right)\\
&=\frac{\partial}{\partial t}\Big|_{t=0}\big(h(\gamma(s,t))-\langle\Phi(\gamma(s,t)),\Ad(\exp(s\eta))\xi\rangle\big).
\end{split}
\end{align}
Now differentiate (\ref{Eq:tDerivative}) with respect to $s$ at $s=0$ to get
\begin{align*}
\begin{split}
0&=\frac{\partial}{\partial s \partial t}\Big|_{s=t=0}\big(h(\gamma(s,t))-\langle\Phi(\gamma(s,t)),\Ad(\exp(s\eta))\xi\rangle\big) \\
&=\frac{\partial}{\partial s \partial t}\Big|_{s=t=0}\big(h(\gamma(s,t))-\langle\Phi(\gamma(s,t)),\xi\rangle\big)-\frac{\partial}{\partial t}\Big|_{t=0}\langle\Phi,[\eta,\xi]\rangle(\gamma(0,t))\\
&=\d^2h^\xi(m)(\eta_M(m),v)-(\d\langle\Phi,[\eta,\xi]\rangle)(m)(v)\\
&=\d^2h^\xi(m)(\eta_M(m),v)-\langle\d\Phi_mv,[\eta,\xi]\rangle\\
&=\d^2h^\xi(m)(\eta_M(m),v),
\end{split}
\end{align*}
where the second equality follows by the product rule and the last equality follows because $v\in\ker\d\Phi_m$. This shows $\left(\d^2h^\xi(m)\right)^\sharp(u)$ vanishes on $\ker\d\Phi_m$, concluding the proof.
\end{proof}

We are now ready to prove the main result of this section. As mentioned in Remark~\ref{Rem:FixedMomentum}, to prove Theorem~\ref{Thm:MROTheorem} it is of no loss of generality to suppose the Hamiltonian vanishes at the relative equilibrium and that the moment is fixed. Under these assumptions, the following lemma reduces Theorem~\ref{Thm:MROTheorem} to Theorem~\ref{Thm:MainTheorem2} by showing that the assumptions on the Hessian in Theorem~\ref{Thm:MROTheorem} reduce to the MGS normal form as in Theorem~\ref{Thm:MainTheorem2}.

\begin{Lemma}\label{Lemma:MGSReduction}
Let $(M,\omega,G,\Phi,h)$ be a Hamiltonian $G$-system with the point $m$ in the mani\-fold~$M$ as a relative equilibrium of the vector field~$\Xi_h$. Suppose that $h(m)=0$ and $\Phi(m)=0$. Denote by $K$ the stabilizer of the point $m$. Let $\left(Z:=G\times^KV,\omega_z,G, \Phi_Z,h_Z\right)$ be the MGS normal form of the given Hamiltonian $G$-system $($see Definition~{\rm \ref{Def:MGSNormalForm})}. Furthermore:
\begin{enumerate}\itemsep=0pt
\item[$1.$] Let $V:=\ffk_r^0\times W_r$ be the slice, where $\ffk_r^0$ is a $K$-invariant neighborhood of the origin in the annihilator $\ffk^0$ of the Lie algebra $\ffk$ of $K$, and $W_r$ is a $K$-invariant neighborhood of the origin in the symplectic slice~$W$ of~$m$.
\item[$2.$] Let the map $j\colon V\hookrightarrow Z$ be the $K$-equivariant embedding defined by $j(\rho,w):=[1,\rho,w]$.
\item[$3.$] Let $\Psi\colon Z\to U_m$ be the symplectomorphism to a neighborhood $U_m$ of the orbit $G\cdot m$ guaranteed by Theorem~{\rm \ref{Thm:MGSNormalForm}}.
\end{enumerate}
Suppose that there exists a velocity $\xi\in\ffg$ of the point $m$ and a linear subspace $U$ of the kernel $\ker\d\Phi_m$ such that
\begin{enumerate}\itemsep=0pt
\item[$1.$] The kernel $\ker\d\Phi_m$ splits as $\ker\d\Phi_m=T_m(G\cdot m)\oplus U$.
\item[$2.$] The Hessian $\d^2h^\xi(m)|_U$ is definite and nondegenerate.
\end{enumerate}
Consider the tangent space $\widetilde W:=T_{[1,0,0]}j\left(\{0\}\times W_r\right)\subseteq T_{[1,0,0]}Z$, then the Hessian
\begin{gather*}
\d^2h_Z^\xi\big([1,0,0]\big)|_{\widetilde W}
\end{gather*}
is well-defined, definite, and nondegenerate.
\end{Lemma}

\begin{proof} To simplify notation throughout this proof, we identify the spaces $W_r$ and $\{0\}\times W_r$ and write
\begin{gather*}
\widetilde W:=T_{[1,0,0]}j(W_r)\subseteq T_{[1,0,0]}Z \qquad\text{and}\qquad \widetilde U:=\d\Psi\big(\widetilde W\big)\subset T_mM.
\end{gather*}
We begin with three observations. First, since the map $\Psi$ pulls back the momentum map on $M$ to the one on $Z$, we get that
\begin{gather}\label{Eq:Observation1}
\d\Psi\big(\ker(\d\Phi_Z)_{[1,0,0]}\big)=\ker\d\Phi_m.
\end{gather}
Second, from the $G$-equivariance of the diffeomorphism $\Psi$ we get that
\begin{gather}\label{Eq:Observation2}
\d\Psi\big(T_{[1,0,0]}(G\cdot [1,0,0])\big)=T_m(G\cdot m).
\end{gather}
Third, note that the linear subspace $\widetilde W$ is $K$-invariant under the $K$-representation it inherits from the action of $G$ on the mani\-fold~$M$. Also recall that $W_r$ is a neighborhood in the symplectic slice~$W$. Thus, in particular, the linear subspace $\widetilde W$ is such that
\begin{gather}\label{Eq:Observation3}
\ker(\d\Phi_Z)_{[1,0,0]}=T_{[1,0,0]}(G\cdot [1,0,0])\oplus \widetilde W.
\end{gather}
Using (\ref{Eq:Observation1}), (\ref{Eq:Observation2}), and~(\ref{Eq:Observation3}) we get that $\widetilde U$ is a $K$-invariant subspace of~$T_mM$ such that
\begin{gather}\label{Eq:Observation4}
\ker\d\Phi_m=T_m(G\cdot m)\oplus \widetilde U.
\end{gather}
Using (\ref{Eq:Observation4}), the assumptions on the space $U$, and Lemma \ref{Lemma:SymplecticSliceIndep}, we obtain that the Hessian $\d^2h^\xi(m)|_{\widetilde U}$ is definite and nondegenerate.

Now observe that the augmented Hamiltonian $h^\xi_Z$ of the pullback $h_Z$ is the pullback of the augmented Hamiltonian $h^\xi$:
\begin{gather*}
h^\xi_Z
=h_Z-\langle\Phi_Z,\xi\rangle
=\Psi^*h-\langle\Psi^*\Phi,\xi\rangle
=\Psi^*h^\xi.
\end{gather*}
Thus, since the point $m$ is a critical point of the augmented Hamiltonian $h^\xi$ we have that
\begin{gather*}
\d h^\xi_Z([1,0,0])=\d\big(\Psi^*h^\xi\big)([1,0,0])=\Psi^*\big(\d h^\xi\big)(m) =\Psi^*(0)=0.
\end{gather*}
Hence, the Hessian $\d^2h^\xi_Z([1,0,0])$ is well-defined.

Finally, by the pullback properties of Hessians (see Remark \ref{Rem:Hessians}), we have that
\begin{gather*}
\d^2h^\xi_Z([1,0,0])|_{\widetilde W} =\d^2\big(\Psi^*h^\xi\big)([1,0,0])|_{\widetilde W} =\Psi^*\big(\d^2h^\xi(m)\big)|_{\widetilde U}.
\end{gather*}
Since the map $\Psi$ is a diffeomorphism, this implies that the Hessian $\d^2h^\xi_Z([1,0,0])|_{\widetilde W}$ is definite and nondegenerate because the Hessian $\d^2h^\xi(m)|_{\widetilde U}$ is definite and nondegenerate.
\end{proof}

\subsection*{Acknowledgements}

The author would like to thank Eugene Lerman for providing guidance throughout this project, for his enduring patience with my many questions, and for the many interesting conversations on the subject. The author would also like to express their gratitude to the anonymous referee for their helpful comments and careful review of the preprint of this paper.

%\cite{PR00,PRW04}

\pdfbookmark[1]{References}{ref}
\LastPageEnding


\begin{thebibliography}{99}
\footnotesize\itemsep=0pt

\bibitem{AMR88}
Abraham R., Marsden J.E., Ratiu T., Manifolds, tensor analysis, and
 applications, \href{https://doi.org/10.1007/978-1-4612-1029-0}{\textit{Applied Mathematical Sciences}}, Vol.~75, 2nd ed.,
 Springer-Verlag, New York, 1988.

\bibitem{CL00}
Chossat P., Lauterbach R., Methods in equivariant bifurcations and dynamical
 systems, \href{https://doi.org/10.1142/4062}{\textit{Advanced Series in Nonlinear Dynamics}}, Vol.~15, World Sci.
 Publ. Co., Inc., River Edge, NJ, 2000.

\bibitem{D96}
Duistermaat J.J., Fourier integral operators, \href{https://doi.org/10.1007/978-0-8176-8108-1}{\textit{Modern Birkh\"auser Classics}},
 Birkh\"auser/Springer, New York, 2011.

\bibitem{DK00}
Duistermaat J.J., Kolk J.A.C., Lie groups, \href{https://doi.org/10.1007/978-3-642-56936-4}{\textit{Universitext}}, Springer-Verlag,
 Berlin, 2000.

\bibitem{F07}
Field M.J., Dynamics and symmetry, \href{https://doi.org/10.1142/9781860948541}{\textit{ICP Advanced Texts in Mathematics}},
 Vol.~3, Imperial College Press, London, 2007.

\bibitem{GS02}
Golubitsky M., Stewart I., The symmetry perspective. From equilibrium to chaos
 in phase space and physical space, \href{https://doi.org/10.1007/978-3-0348-8167-8}{\textit{Progress in Mathematics}}, Vol.~200, Birkh\"auser Verlag, Basel, 2002.

\bibitem{GuK02}
Guillemin V., Ginzburg V., Karshon Y., Moment maps, cobordisms, and
 {H}amiltonian group actions, \href{https://doi.org/10.1090/surv/098}{\textit{Mathe\-ma\-tical Surveys and Monographs}},
 Vol.~98, Amer. Math. Soc., Providence, RI, 2002.

\bibitem{GuS84}
Guillemin V., Sternberg S., A normal form for the moment map, in Differential
 Geometric Methods in Mathematical Physics ({J}erusalem, 1982), \textit{Math.
 Phys. Stud.}, Vol.~6, Reidel, Dordrecht, 1984, 161--175.

\bibitem{H09}
Hepworth R., Vector fields and flows on differentiable stacks, \textit{Theory
 Appl. Categ.} \textbf{22} (2009), 542--587, \href{https://arxiv.org/abs/0810.0979}{arXiv:0810.0979}.

\bibitem{KMS93}
Kol\'a\v{r} I., Michor P.W., Slov\'ak J., Natural operations in differential
 geometry, \href{https://doi.org/10.1007/978-3-662-02950-3}{Springer-Verlag}, Berlin, 1993.

\bibitem{K90}
Krupa M., Bifurcations of relative equilibria, \href{https://doi.org/10.1137/0521081}{\textit{SIAM~J. Math. Anal.}}
 \textbf{21} (1990), 1453--1486.

\bibitem{L15}
Lerman E., Invariant vector fields and groupoids, \href{https://doi.org/10.1093/imrn/rnu170}{\textit{Int. Math. Res. Not.}}
 \textbf{2015} (2015), 7394--7416, \href{https://arxiv.org/abs/1307.7733}{arXiv:1307.7733}.

\bibitem{LS98}
Lerman E., Singer S.F., Stability and persistence of relative equilibria at
 singular values of the moment map, \href{https://doi.org/10.1088/0951-7715/11/6/012}{\textit{Nonlinearity}} \textbf{11} (1998),
 1637--1649.

\bibitem{Ma85}
Marle C.-M., Mod\`ele d'action hamiltonienne d'un groupe de {L}ie sur une
 vari\'et\'e symplectique, \textit{Rend. Sem. Mat. Univ. Politec. Torino}
 \textbf{43} (1985), 227--251.

\bibitem{M92}
Marsden J.E., Lectures on mechanics, \href{https://doi.org/10.1017/CBO9780511624001}{\textit{London Mathematical Society
 Lecture Note Series}}, Vol.~174, Cambridge University Press, Cambridge, 1992.

\bibitem{MRO15}
Montaldi J., Rodr\'{\i}guez-Olmos M., Hamiltonian relative equilibria with
 continuous isotropy, \href{https://arxiv.org/abs/1509.04961}{arXiv:1509.04961}.

\bibitem{MRO11}
Montaldi J., Rodr\'{\i}guez-Olmos M., On the stability of {H}amiltonian
 relative equilibria with non-trivial isotropy, \href{https://doi.org/10.1088/0951-7715/24/10/007}{\textit{Nonlinearity}}
 \textbf{24} (2011), 2777--2783, \href{https://arxiv.org/abs/1011.1130}{arXiv:1011.1130}.

\bibitem{OR99}
Ortega J.-P., Ratiu T.S., Stability of {H}amiltonian relative equilibria,
 \href{https://doi.org/10.1088/0951-7715/12/3/315}{\textit{Nonlinearity}} \textbf{12} (1999), 693--720.

\bibitem{OR04}
Ortega J.-P., Ratiu T.S., Momentum maps and {H}amiltonian reduction,
 \href{https://doi.org/10.1007/978-1-4757-3811-7}{\textit{Progress in Mathematics}}, Vol.~222, Birkh\"auser Boston, Inc.,
 Boston, MA, 2004.

\bibitem{Pa61}
Palais R.S., On the existence of slices for actions of non-compact {L}ie
 groups, \href{https://doi.org/10.2307/1970335}{\textit{Ann. of Math.}} \textbf{73} (1961), 295--323.

\bibitem{P91}
Patrick G.W., Two axially symmetric coupled rigid bodies: {R}elative
 equilibria, stability, bifurcations, and a momentum preserving symplectic
 integrator, Ph.D.~Thesis, {U}niversity of California, Berkeley, 1991.

\bibitem{P95}
Patrick G.W., Relative equilibria of {H}amiltonian systems with symmetry:
 linearization, smoothness, and drift, \href{https://doi.org/10.1007/BF01212907}{\textit{J.~Nonlinear Sci.}} \textbf{5}
 (1995), 373--418.

\bibitem{RWL02}
Roberts M., Wulff C., Lamb J.S.W., Hamiltonian systems near relative
 equilibria, \href{https://doi.org/10.1006/jdeq.2001.4045}{\textit{J.~Differential Equations}} \textbf{179} (2002), 562--604.

\bibitem{RdS97}
Roberts R.M., de~Sousa~Dias M.E.R., Bifurcations from relative equilibria of
 {H}amiltonian systems, \href{https://doi.org/10.1088/0951-7715/10/6/015}{\textit{Non\-li\-nearity}} \textbf{10} (1997), 1719--1738.

\bibitem{WR01}
Wulff C., Patrick G., Roberts M., Stability of Hamiltonian relative equilibria
 by energy methods, in Symmetry and Perturbation Theory (Cala Gononoe, 2001),
\href{https://doi.org/10.1142/9789812794543_0028}{World Sci. Publ.}, River Edge, NJ, 2001,
 214--221.

\end{thebibliography}
\end{document}